\documentclass[leqno,12pt]{amsart}
\usepackage{a4wide}
\usepackage{amscd}
\usepackage{amsmath}
\usepackage{amsfonts}
\usepackage{amssymb}
\usepackage{multirow}
\usepackage{color}
\usepackage{hyperref}

\newcommand{\RR}{\mathbb{R}}
\newcommand{\CC}{\mathbb{C}}

\newcommand{\ca}{{\mathfrak{a}}}
\newcommand{\cA}{{\mathfrak{A}}}

\newcommand{\cd}{{\mathfrak{d}}}

\newcommand{\cf}{{\mathfrak{f}}}

\newcommand{\cg}{{\mathfrak{g}}}

\newcommand{\ch}{{\mathfrak{h}}}

\newcommand{\ck}{{\mathfrak{k}}}

\newcommand{\cl}{{\mathfrak{l}}}

\newcommand{\cm}{{\mathfrak{m}}}

\newcommand{\cn}{{\mathfrak{n}}}

\newcommand{\co}{{\mathfrak{o}}}

\newcommand{\cp}{{\mathfrak{p}}}

\newcommand{\cq}{{\mathfrak{q}}}

\newcommand{\cs}{{\mathfrak{s}}}

\newcommand{\cu}{{\mathfrak{u}}}

\newcommand{\cX}{{\mathfrak{X}}}

\newcommand{\calC}{{\mathcal{C}}}

\newcommand{\calQ}{{\mathcal{Q}}}

\newcommand{\calK}{{\mathcal{K}}}

\newcommand{\id}{{\rm {id}}}

\newcommand{\Ric}{{\rm {Ric}}}

\newcommand{\tr}{{\rm {tr}}}

\newcommand{\Ad}{{\rm {Ad}}}
\newcommand{\Aut}{{\rm {Aut}}}
\newcommand{\ad}{{\rm {ad}}}

\newcommand{\rk}{{\rm {rk}}}
\newcommand{\Exp}{{\rm {Exp}}}
\newcommand{\End}{{\rm {End}}}

\newcommand{\sC}{\Gamma(\calC)}

\newtheorem{thm}{Theorem}[section]

\newtheorem{prop}[thm]{Proposition}
\newtheorem{cor}[thm]{Corollary}

\newtheorem{re}[thm]{Remark}

\allowdisplaybreaks

\numberwithin{equation}{section}

\begin{document}

\title[Foliated Hopf hypersurfaces in complex hyperbolic quadrics]{Foliated Hopf hypersurfaces\\ in complex hyperbolic quadrics}

\author{J\"{u}rgen Berndt}
\address{Department of Mathematics, King's College London, London, WC2R 2LS, United Kingdom}
\email{jurgen.berndt@kcl.ac.uk}
\thanks{}


\begin{abstract}
This paper deals with a limiting case motivated by contact geometry. The limiting case of a tensorial characterization of contact hypersurfaces in K\"{a}hler manifolds leads to Hopf hypersurfaces whose maximal complex subbundle of the tangent bundle is integrable. It is known that in non-flat complex space forms and in complex quadrics such real hypersurfaces do not exist, but the existence problem in other irreducible K\"{a}hler manifolds is open. In this paper we construct explicitly a one-parameter family of homogeneous Hopf hypersurfaces, whose maximal complex subbundle of the tangent bundle is integrable, in a Hermitian symmetric space of non-compact type and rank two. These are the first known examples of such real hypersurfaces in irreducible K\"{a}hler manifolds.
\end{abstract}

\maketitle 

\thispagestyle{empty}

\section {Introduction}

We start with the motivation for this paper. A contact manifold is a smooth odd-dimensional manifold $M$ together with a $1$-form $\eta$ on $M$ satisfying $\eta \wedge (d\eta)^{n-1} \neq 0$, where $\dim_\RR(M) = 2n-1$. Such a $1$-form $\eta$ is called a contact form. The kernel of $\eta$ defines a hyperplane distribution $\calC$ on $M$, the so-called contact distribution.  The contact condition $\eta \wedge (d\eta)^{n-1} \neq 0$ means that the maximal possible dimension of a submanifold of $M$ all of whose tangent spaces are contained in $\calC$ is equal to $n-1$. The contact condition  therefore is a measure for maximal non-integrability of $\calC$.

Let $\bar{M}$ be a K\"{a}hler manifold with K\"{a}hler structure $J$, K\"{a}hler metric $g$ and $n = \dim_\CC(\bar{M}) \geq 2$. Let $M$ be a real hypersurface in $\bar{M}$ and $(\phi,\xi,\eta,g)$ be the the induced almost contact metric structure on $M$ (see Section \ref{mcs}). The subbundle $\calC = \ker(\eta) = TM \cap J(TM)$ of $TM$ is the maximal complex subbundle of the tangent bundle $TM$. The real hypersurface $M$ is said to be a contact hypersurface if there exists an everywhere non-zero smooth function $f : M \to \RR$ so that $d\eta = 2f \omega$, where $\omega$ is the fundamental $2$-form on $M$ defined by $\omega(X,Y) = g(\phi X,Y)$ for all $X,Y \in \cX(M)$. The fundamental $2$-form $\omega$ is always closed, which implies $\eta \wedge d\eta^{n-1} = (2f)^{n-1}(\eta \wedge \omega^{n-1}) \neq 0$ if $M$ is a contact hypersurface. Thus every contact hypersurface in a K\"{a}hler manifold is a contact manifold. In this situation the maximal complex subbundle $\calC$ of the tangent bundle of the contact hypersurface coincides with the contact distribution. A natural problem is to determine the contact hypersurfaces in K\"{a}hler manifolds.

The first systematic study of contact hypersurfaces in K\"{a}hler manifolds was carried out by Okumura \cite{Ok66}. Okumura proved the following very useful characterization of contact hypersurfaces in K\"{a}hler manifolds: A real hypersurface $M$ in a K\"{a}hler manifold $\bar{M}$ is a contact hypersurface if and only if there exists an every non-zero smooth function $f : M \to \RR$ so that that the shape operator $A$ of $M$ and the structure tensor field $\phi$ satisfy $A\phi + \phi A = 2f \phi$. It is not difficult to prove that the function $f$ is constant when $n > 2$ (see \cite{BS22}, Proposition 3.5.4). Starting from Okumura's work, contact hypersurfaces were classified in various Hermitian symmetric spaces (see \cite{BS22} for an overview). The motivation for this paper is to understand the limiting case $f = 0$. We will show (see Proposition \ref{HopfCint}) that the limiting case $f = 0$ characterizes Hopf hypersurfaces  in K\"{a}hler manifolds for which the maximal complex subbundle $\calC$ is integrable. For the concept of Hopf hypersurfaces see Section \ref{mcs}.

The totally geodesic real hypersurface $\RR^{2n-1}$ in the complex Euclidean space $\CC^n$ is an elementary example of a Hopf hypersurface whose maximal complex subbundle $\calC$ is integrable. In contrast, it is quite remarkable and not obvious that in non-flat complex space forms there are no Hopf hypersurfaces whose maximal complex subbundle $\calC$ is integrable. This is not difficult to prove for the complex projective space $\CC P^n(c)$ with the Fubini-Study metric of constant holomorphic sectional curvature $c > 0$, but the proof is quite involved for the complex hyperbolic space $\CC H^n(c)$ with the Bergman metric of constant holomorphic sectional curvature $c < 0$. A detailed discussion of these two cases can be found in Section 2 of \cite{NR97}. These non-existence results raise the existence question for other irreducible K\"{a}hler manifolds. In \cite{BS22}, the geometry of real hypersurfaces in some irreducible Hermitian symmetric spaces of rank $2$ was investigated. One of these Hermitian symmetric spaces is the Grassmann manifold $SO_{2+n}/(SO_2\times SO_n)$ of oriented $2$-planes in $\RR^{2+n}$, which is isometric to the complex quadric $Q^n$ in $\CC P^{n+1}(c)$ (with a suitable normalization of the metric). From the investigations in \cite{BS22}, Section 6.4, we can conclude that there are no Hopf hypersurfaces in this Hermitian symmetric space for which $\calC$ is integrable. 

In this paper we investigate the existence question in the dual Hermitian symmetric space of non-compact type, the complex hyperbolic quadric ${Q^n}^* = SO^o_{2,n}/(SO_2\times SO_n)$. Surprisingly, we can construct a one-parameter family of pairwise non-congruent homogeneous Hopf hypersurfaces in ${Q^n}^*$ whose maximal complex subbundle $\calC$ is integrable. 

\begin{thm}\label{mainthm}
There exists a one-parameter family $M^{2n-1}_\alpha$, $0 \leq \alpha < \infty$, of (pairwise non-congruent) homogeneous Hopf hypersurfaces, whose maximal complex subbundle of the tangent bundle is integrable, in the Hermitian symmetric space ${Q^n}^* = SO^o_{2,n}/(SO_2\times SO_n)$, $n \geq 3$. 
\end{thm}

We give a brief geometric description of these real hypersurfaces. We normalize the Riemannian metric on ${Q^n}^*$ so that the minimum of the sectional curvature is equal to $-4$. The complex hyperbolic quadric ${Q^n}^*$ is equipped with a circle bundle $\cA_0$ of real structures (see Section \ref{tchq}). This circle bundle determines a maximal $\cA_0$-invariant subbundle $\calQ$ of the tangent bundle $TM$ of $M$. The maximal Satake compactification of ${Q^n}^* = SO^o_{2,n}/(SO_2\times SO_n)$ has two boundary components of rank $1$, namely a complex hyperbolic line $B_1 \cong \CC H^1(-4)$ of constant (holomorphic) sectional curvature $-4$ and a real hyperbolic space $B_2 \cong\RR H^{n-2}(-2)$ of constant sectional curvature $-2$. It is an interesting fact that all non-zero tangent vectors of $B_1$ are singular tangent vectors of ${Q^n}^*$ of a particular type (that is, tangent vectors that are contained in more than one maximal flat of ${Q^n}^*$).
In \cite{BT13}, we developed a technique, the so-called canonical extension method, for extending isometric actions on boundary components of irreducible Riemannian symmetric spaces of non-compact type to isometric actions on the entire symmetric space. This method can be used to extend submanifolds in boundary components. By extending a point in the boundary component $B_1 \cong \CC H^1(-4)$ we obtain an isometric embedding $P^{n-1}$ of the complex hyperbolic space $ \CC H^{n-1}(-4)$ into ${Q^n}^*$ as a homogeneous complex hypersurface. This construction will be explained in detail in Section \ref{hcs}, where we will also investigate the geometry of this homogeneous complex hypersurface.

This homogeneous complex hypersurface $P^{n-1}$ will appear as the integral manifolds of the integrable distribution $\calC$ in our examples. The Langlands decomposition of the parabolic subgroup of $SO^o_{2,n}$ with boundary component $B_1 \cong \CC H^1(-4)$ induces a horospherical decomposition $B_1 \times \RR \times H^{2n-3}$ of ${Q^n}^*$, where $H^{2n-3}$ is the $(2n-3)$-dimensional Heisenberg group with $1$-dimensional center. The product $\RR \times H^{2n-3}$ corresponds to the homogeneous complex hypersurface $P^{n-1} \cong \CC H^{n-1}(-4)$. 
Now take any complete curve $\gamma$ in the boundary component $B_1 \cong \CC H^1(-4)$ with constant geodesic curvature $\alpha \geq 0$. The curve $\gamma$ is a geodesic in $\CC H^1(-4)$ if $\alpha = 0$, an equidistant curve to a geodesic  in $\CC H^1(-4)$ if $0 < \alpha < 2$, a horocycle if $\alpha = 2$, or a closed circle in $\CC H^1(-4)$ if $2 < \alpha < \infty$. Sliding the homogeneous complex hypersurface $P^{n-1} \cong \CC H^{n-1}(-4)$ along the curve $\gamma$ in a suitable way, we obtain a homogeneous Hopf hypersurface $M^{2n-1}_\alpha$ in ${Q^n}^*$ whose maximal complex subbundle $\calC$ is integrable. We will see that the homogeneous real hypersurface $M^{2n-1}_\alpha$ has constant principal curvatures $\alpha$, $0$, $+1$, $-1$ with multiplicities $1$, $2$, $n-1$, $n-1$, respectively. The principal curvature space $T_\alpha$ is equal to the orthogonal complement $\calC^\perp$ of $\calC$ in $TM$. The principal curvature space $T_0$  is equal to the orthogonal complement $\calC \ominus \calQ$ of $\calQ$ in $\calC$. The principal curvature spaces $T_1$ and $T_{-1}$ span $\calQ$, are mapped into each other by the structure tensor field $\phi$, and are equal to the $\pm 1$-eigenspaces of the restriction to $\calQ$ of a suitable real structure in $\cA_0$. The hypersurfaces $M^{2n-1}_\alpha$ will in fact be constructed through an algebraic process, and the ``sliding'' description is a geometric interpretation of this algebraic construction, which will be explained thoroughly during the construction process. The homogeneous real hypersurface $M^{2n-1}_\alpha$ is diffeomeorphic to $\RR^{2n-1}$ for $0 \leq \alpha \leq 2$ and diffeomorphic to $S^1 \times \RR^{2n-2}$ for $2 < \alpha < \infty$. 

We point out that none of the homogeneous real hypersurfaces $M^{2n-1}_\alpha$ in Theorem \ref{mainthm} arises as a limit of contact hypersurfaces in ${Q^n}^*$. The classification of contact hypersurfaces in ${Q^n}^*$ can be found in Section 7.8 of \cite{BS22}. For every real number $f > 0$ there exists, up to isometric congruence, a unique connected complete contact hypersurface $\tilde{M}^{2n-1}_f$ in ${Q^n}^*$ satisfying $A\phi + \phi A = 2f\phi$. This family $\tilde{M}^{2n-1}_f$ of contact hypersurfaces collapses to a totally geodesic complex embedding of the complex hyperbolic quadric ${Q^{n-1}}^*$ into ${Q^n}^*$ when taking the limit $f \to 0$, and so $\lim_{f\to 0} \tilde{M}^{2n-1}_f = {Q^{n-1}}^*$ is not a real hypersurface.

The paper is organized as follows. In Section \ref{mcs} we introduce basic concepts from almost contact metric geometry in K\"{a}hler manifolds and provide characterizations of Hopf hypersurfaces and of real hypersurfaces satisfying $A\phi + \phi A = 0$. In Section \ref{tchq} we present two models for the complex hyperbolic quadric ${Q^n}^* = SO^o_{2,n}/(SO_2\times SO_n)$. The first one is the standard symmetric space model, the second one is the solvable Lie group model originating from an Iwasawa decomposition of $SO^o_{2,n}$. The interplay between both models allows us to switch between geometric and algebraic interpretations of relevant concepts. In Section \ref{hcs} we construct the isometric embedding of the complex hyperbolic space $\CC H^{n-1}(-4)$ as a homogeneous complex hypersurface $P^{n-1}$ in ${Q^n}^*$ and discuss aspects of the geometry of this embedding. The homogeneous real hypersurfaces $M^{2n-1}_\alpha$ ($2 < \alpha$) will be constructed in Section \ref{thch} as the tubes around the homogeneous complex hypersurface $P^{n-1}$ in ${Q^n}^*$. In Section \ref{mhHs} we use the theory of parabolic subalgebras of real semisimple Lie algebras for the construction of the minimal homogeneous real hypersurface $M^{2n-1}_0$. In Section \ref{eqdihyp} we construct the homogeneous real hypersurfaces $M^{2n-1}_\alpha$ ($0 < \alpha < 2$) as the equidistant hypersurfaces to $M^{2n-1}_0$.  The homogeneous real hypersurface $M^{2n-1}_2$ will be constructed in Section \ref{canexthoro} as the canonical extension of a horocycle in the boundary component $B_1 \cong \CC H^1(-4)$. We will also investigate the geometry of the homogeneous real hypersurfaces $M^{2n-1}_\alpha$ in the corresponding sections. In Section \ref{curvature} we investigate the curvature of the homogeneous real hypersurfaces $M^{2n-1}_\alpha$.

\section{The maximal complex subbundle of the tangent bundle}\label{mcs}

Let $\bar{M}$ be a K\"{a}hler manifold with K\"{a}hler structure $J$ and K\"{a}hler metric $g$. We always assume $n = \dim_\CC(\bar{M}) \geq 2$. Let $M$ be a real hypersurface in $\bar{M}$. We will denote the induced Riemannian metric on $M$ also by $g$. The Levi Civita covariant derivative of $\bar{M}$ and $M$ is denoted by $\bar\nabla$ and $\nabla$, respectively. The Lie algebra of smooth vector fields on $M$ is denoted by $\cX(M)$.

Let $\zeta$ be a (local) unit normal vector field on $M$. We denote by $A = A_\zeta$ the shape operator of $M$ with respect to $\zeta$. The unit vector field 
\[
\xi = -J\zeta
\]
is the Reeb vector field on $M$. The flow of the Reeb vector field $\xi$ is the Reeb flow on $M$. We define a $1$-form $\eta$ on $M$ by
\[
\eta(X) = g(X,\xi)
\]
for all $X \in \cX(M)$ and a skew-symmetric tensor field $\phi$ on $M$ by decomposing $JX$ into its tangential component $\phi X$ and its normal component $g(JX,\zeta)\zeta$, that is, 
\[
JX = \phi X + g(JX,\zeta)\zeta  = \phi X + \eta(X)\zeta 
\]
for all $X \in \cX(M)$. The $1$-form $\eta$ is the almost contact form on $M$ and the skew-symmetric tensor field $\phi$ is the structure tensor field on $M$. The quadruple $(\phi,\xi,\eta,g)$ is the induced almost contact metric structure on $M$. Note that 
\[
\eta(\xi) = 1,\ \phi\xi  = 0 \mbox{ and } 
 \phi^2 X  = -X + \eta(X)\xi 
\]
for all $X \in \cX(M)$.
Using the K\"{a}hler property $\bar\nabla J = 0$ and the Weingarten formula we obtain
\[
0 = (\bar\nabla_XJ)\zeta = \bar\nabla_X J\zeta - J\bar\nabla_X\zeta = -\bar\nabla_X\xi + JAX
\]
for all $X \in \cX(M)$. The tangential component of this equation induces the useful equation
\[
\nabla_X\xi = \phi AX
\]
for all $X \in \cX(M)$.

The subbundle
\[
\calC = \ker(\eta) = TM \cap J(TM)
\]
of the tangent bundle $TM$ of $M$ is the maximal complex subbundle of $TM$. We denote by $\sC$ the set of all vector fields $X$ on $M$ with values in $\calC$, that is,
\begin{align*}
\sC & = \{ X \in \cX(M) : X_p \in \calC_p \mbox{ for all } p \in M\} \\
& = \{ X \in \cX(M) : \eta(X) = 0 \} .
\end{align*}

The real hypersurface $M$ is called a Hopf hypersurface if the Reeb flow on $M$ is a geodesic flow, that is, if the integral curves of the Reeb vector field $\xi$ are geodesics in $M$. We have the following characterization of Hopf hypersurfaces.

\begin{prop}
Let $M$ be a real hypersurface in a K\"{a}hler manifold $\bar{M}$ with induced almost contact metric structure $(\phi,\xi,\eta,g)$. The following statements are equivalent:
\begin{enumerate}
\item[\rm (i)] $M$ is a Hopf hypersurface in $\bar{M}$;
\item[\rm (ii)] $\nabla_\xi\xi = 0$;
\item[\rm (iii)] The Reeb vector field $\xi$ is a principal curvature vector of $M$ at every point;
\item[\rm (iv)] The maximal complex subbundle $\calC$ of $TM$ is invariant under the shape operator $A$ of $M$, that is, $A\calC \subseteq \calC$.
\end{enumerate}
\end{prop}

\begin{proof}
Let $p \in M$ and $c : I \to M$ be an integral curve of the Reeb vector field $\xi$ with $0 \in I$ and $c(0) = p$. Then we have $\nabla_{\xi_p}\xi = \nabla_{\dot{c}(0)}\xi = (\xi \circ c)'(0) = \dot{c}'(0)$.
If $M$ is a Hopf hypersurface, then we have $\dot{c}'(0) = 0$ by definition and therefore $\nabla_{\xi_p}\xi = 0$. Since this holds at any point $p \in M$, we obtain $\nabla_\xi\xi = 0$. Conversely, if $\nabla_\xi\xi = 0$, then $\dot{c}' = \nabla_{\dot{c}}\xi = \nabla_{\xi \circ c}\xi = 0$ for any integral curve $c$ of $\xi$. Thus any integral curve of $\xi$ is a geodesic in $M$ and hence $M$ is a Hopf hypersurface. This establishes the equivalence of (i) and (ii)

The kernel $\ker(\phi)$ of the structure tensor field $\phi$ is spanned by the Reeb vector field, that is,
$\ker(\phi) = \RR\xi$.
Since $\nabla_\xi\xi = \phi A\xi$, we therefore see that $\nabla_\xi\xi = 0$ if and only if $A\xi \in \RR\xi$, which shows that (ii) and (iii) are equivalent.

We have the orthogonal decomposition $TM = \calC \oplus \RR\xi$.
Since the shape operator $A$ is self-adjoint, the equivalence of (iii) and (iv) is obvious.
\end{proof}

The next result provides a characterization of real hypersurfaces when taking the limit $f \to 0$ in Okumura's characterization $A\phi + \phi A = 2f\phi$ of contact hypersurfaces in K\"{a}hler manifolds.

\begin{prop} \label{HopfCint}
Let $M$ be a real hypersurface in a K\"{a}hler manifold $\bar{M}$ with induced almost contact metric structure $(\phi,\xi,\eta,g)$. The following statements are equivalent:
\begin{enumerate}
\item[\rm (i)] The almost contact form  $\eta$ is closed, that is, $d\eta = 0$.
\item[\rm (ii)] The shape operator $A$ of $M$ and the structure tensor field $\phi$ satisfy 
\[
A \phi + \phi A = 0.
\]
\item[\rm (iii)] The real hypersurface $M$ is a Hopf hypersurface and the maximal complex subbundle $\calC$ of $TM$ is integrable.
\end{enumerate}
\end{prop}

\begin{proof}
Using the equation $\nabla_X\xi = \phi AX$, the exterior derivative $d\eta$ of $\eta$ is
\begin{align*} 
d\eta(X,Y) & = d(\eta(Y))(X) - d(\eta(X))(Y) - \eta([X,Y])  \\ 
& = Xg(Y,\xi) - Yg(X,\xi) - g([X,Y],\xi) \\ 
& = g(\nabla_XY,\xi) + g(Y,\nabla_X\xi) - g(\nabla_YX,\xi) - g(X,\nabla_Y\xi) - g([X,Y],\xi)  \\ 
& = g(Y,\phi A X) - g(X, \phi A Y)   \\
& =  g((A\phi + \phi A)X,Y)
\end{align*}
for all $X,Y \in \cX(M)$. It follows that $\eta$ is closed if and only if $A\phi + \phi A = 0$, which shows that (i) and (ii) are equivalent.

The above calculations imply that for $X,Y \in \sC$ we have
\[
\eta([X,Y]) = - d\eta(X,Y)  = -g((A\phi + \phi A)X,Y).
\]
It follows that the distribution $\calC$ is involutive if and only if $g((A\phi + \phi A)X,Y) = 0$ holds for all $X,Y \in \sC$. We have $g((A\phi + \phi A)\xi,Y) = g(\phi A\xi, Y) = 0$ for all $Y \in \sC$ if and only if $A\xi \in \RR\xi$, that is, if and only if $M$ is a Hopf hypersurface. We always have $g((A\phi + \phi A)\xi,\xi) = 0$. Using Frobenius Theorem we can now conclude the equivalence of (ii) and (iii). 
\end{proof}

\section{The complex hyperbolic quadric}\label{tchq}

The complex hyperbolic quadric is the Riemannian symmetric space
\[
{Q^n}^* = SO^o_{2,n}/(SO_2\times SO_n),\ n \geq 1,
\]
where $SO^o_{2,n}$ denotes the identity component of the indefinite special orthogonal group $SO_{2,n}$ and $SO_2\times SO_n$ is embedded canonically into $SO^o_{2,n}$. 
The complex hyperbolic quadric $SO^o_{2,n}/(SO_2\times SO_n)$ is the non-compact dual symmetric space of the complex quadric $SO_{2+n}/(SO_2\times SO_n)$. We put $G = SO^o_{2,n}$, $K = SO_2\times SO_n$, and denote by $o \in {Q^n}^*$ the ``base point'' $I_{2+n}K$ of the homogeneous space $G/K$, where $I_{2+n} \in G$ is the identity $((2+n) \times (2+n))$-matrix. Then $K$ is the isotropy group of $G$ at $o$. We now describe the construction of the complex hyperbolic quadric as a Riemannian symmetric space in some more detail. 

We denote by $M_{2,n}(\RR)$ the real vector space of $(2 \times n)$-matrices with real coefficients. Let 
\[
\cg = \cs\co_{2,n} 
= \left\{ \begin{pmatrix}
A_1 & B \\ B^\top & A_2 
\end{pmatrix} : A_1 \in \cs\co_2,\ A_2\in\cs\co_n,\ B \in M_{2,n}(\RR)
 \right\}
\] 
be the Lie algebra of $G = SO^o_{2,n}$ and 
\[
\ck = \cs\co_2 \oplus \cs\co_n
= \left\{ \begin{pmatrix}
A_1 & 0_{2,n} \\ 0_{n,2} & A_2 
\end{pmatrix} : A_1 \in \cs\co_2,\ A_2\in\cs\co_n
 \right\}
\] 
 be the Lie algebra of $K = SO_2\times SO_n$. Let 
\[
B : \cg \times \cg \to \RR\ ,\ (X,Y) \mapsto \tr(\ad(X)\ad(Y)) = n\tr(XY)
\]
be the Killing form of $\cg$ and 
\[
\cp = \left\{ \begin{pmatrix}
0_{2,2} & B \\ B^\top & 0_{n,n} 
\end{pmatrix} :  B \in M_{2,n}(\RR)
 \right\}
\]
be the orthogonal complement of $\ck$ in $\cg$ with respect to $B$. The resulting decomposition $\cg = \ck \oplus \cp$ is a Cartan decomposition of $\cg$.  We identify the tangent space $T_o{Q^n}^*$ of ${Q^n}^*$ at $o$ with $\cp$ in the usual way. 

The Cartan involution $\theta \in \Aut(\cg)$ on $\cg$ is given by 
\[
\theta(X) = I_{2,n} X I_{2,n} \mbox{ with } I_{2,n} = \begin{pmatrix} -I_2 & 0_{2,n} \\ 0_{n,2} & I_n \end{pmatrix},
\] 
where $I_2$ and $I_n$ is the identity $(2 \times 2)$-matrix and $(n \times n)$-matrix respectively. Then 
\[
B_\theta : \cg \times \cg \to \RR\ ,\ (X,Y) = -B(X,\theta(Y)) 
\]
is a positive definite $\Ad(K)$-invariant inner product on $\cg$. The Cartan decomposition $\cg = \ck \oplus \cp$ is orthogonal with respect to $B_\theta$. The restriction of $B_\theta$ to $\cp \times \cp$ induces a $G$-invariant Riemannian metric $g_{B_\theta}$ on ${Q^n}^*$, which is often referred to as the standard homogeneous metric on ${Q^n}^*$. The complex hyperbolic quadric $({Q^n}^*,g_{B_\theta})$ is an Einstein manifold with Einstein constant $-\frac{1}{2}$ (see \cite{WZ85} and use duality between Riemannian symmetric spaces of compact type and of non-compact type). We renormalize the standard homogeneous metric $g_{B_\theta}$ so that the Einstein constant of the renormalized Riemannian metric $g$ is equal to $-2n$, that is,
\[
g_{B_\theta} = 4ng.
\]
This renormalization implies that the minimum of the sectional curvature of $({Q^n}^*,g)$ is equal to $-4$. Note that $({Q^1}^*,g)$ is isometric to the complex hyperbolic line $\CC H^1(-4)$ and $({Q^2}^*,g)$ is isometric to the Riemannian product $\CC H^1(-4) \times \CC H^1(-4)$ of two complex hyperbolic lines. For $n \geq 3$, $({Q^n}^*,g)$ is an irreducible Riemannian symmetric space of non-compact type and rank $2$. We assume $n \geq 3$ in the following.

The Lie algebra $\ck$ decomposes orthogonally into $\ck  = \cs\co_2 \oplus \cs\co_n$. The first factor $\cs\co_2$ is the $1$-dimensional center of $\ck$. The adjoint action of 
\[ 
Z =  \begin{pmatrix}
0 & -1 & 0 & \cdots & 0 \\
1 & 0 & 0 & \cdots & 0 \\
0 & 0 & 1 & \cdots & 0 \\
\vdots & \vdots & \vdots  & \ddots & \vdots\\
0 & 0 & 0 &  \cdots & 1 \\
\end{pmatrix} \in SO_2 \subset SO_2\times SO_n = K
\]
on $\cp$ induces a K\"{a}hler structure $J$ on ${Q^n}^*$. In this way $({Q^n}^*,g,J)$ becomes a Hermitian symmetric space.

We define
\[
 c_0 =  \begin{pmatrix}
1 & 0 & 0 & \cdots & 0 \\
0 & -1 & 0 & \cdots & 0 \\
0 & 0 & 1 & \cdots & 0 \\
\vdots & \vdots & \vdots  & \ddots & \vdots\\
0 & 0 & 0 &  \cdots & 1 \\
\end{pmatrix}
\in O_2\times SO_n . 
 \]
Note that $c_0 \not\in K$, but $c_0$ is in the isotropy group at $o$ of the full isometry group of $({Q^n}^*,g)$. The adjoint transformation $\Ad(c_0)$ leaves $\cp$ invariant and $C_0 = \Ad(c_0)|_\cp$ is an anti-linear involution on $\cp \cong T_o{Q^n}^*$ satisfying $C_0J + JC_0 = 0$. In other words, $C_0$ is a real structure on $T_o{Q^n}^*$. The involution $C_0$ commutes with $\Ad(g)$ for all $g \in SO_n \subset K$ but not for all $g \in K$. More precisly, for $g = (g_1,g_2) \in K$ with $g_1 \in SO_2$ and $g_2 \in SO_n$, say $g_1 = \left( \begin{smallmatrix} \cos(\varphi) & -\sin(\varphi) \\ \sin(\varphi) & \cos(\varphi) \end{smallmatrix} \right)$ with $\varphi \in \RR$, so that $\Ad(g_1)$ corresponds to multiplication with the complex number $\mu = e^{i\varphi}$, we have
\[
C_0 \circ \Ad(g) = \mu^{-2} \Ad(g) \circ C_0  . 
\]
It follows that we have a circle of real structures
\[
\{\cos(\varphi)C_0 + \sin(\varphi)JC_0 : \varphi \in \RR \} . 
\]
This set is $\Ad(K)$-invariant and therefore generates an $\Ad(G)$-invariant $S^1$-subbundle $\cA_0$ of the endomorphism bundle $\End(T{Q^n}^*)$, consisting of real structures (or conjugations) on the tangent spaces of ${Q^n}^*$.
This $S^1$-bundle naturally extends to an $\Ad(G)$-invariant vector subbundle $\cA$ of $\End(T{Q^n}^*)$ with $\rk(\cA) = 2$, which is parallel with respect to the induced connection on $\End(T{Q^n}^*)$. For any real structure $C \in \cA_0$ the tangent line to the fibre of $\cA$ through $C$ is spanned by $JC$. For every $p \in {Q^n}^*$ and real structure $C \in \cA_p$ we have an orthogonal decomposition
\[
T_p{Q^n}^* = V(C) \oplus JV(C) 
\]
into two totally real subspaces of $T_p{Q^n}^*$. Here $V(C)$ and $JV(C)$ are the $(+1)$- and $(-1)$-eigenspaces of $C$, respectively. 
By construction, we have
\[
V(C_0) = \left\{ 
\begin{pmatrix}
0 & 0 & u_1 & \cdots & u_n \\
0 & 0 & 0 & \cdots & 0 \\
u_1 & 0 & 0 & \cdots & 0 \\
\vdots & \vdots & \vdots  & \ddots & \vdots\\
u_n & 0 & 0 & \cdots & 0 
\end{pmatrix} : u \in \RR^n
\right\}
\]
and
\[
JV(C_0) = \left\{ 
\begin{pmatrix}
0 & 0 & 0 & \cdots & 0 \\
0 & 0 & v_1 & \cdots & v_n \\
0 & v_1 & 0 & \cdots & 0 \\
\vdots & \vdots & \vdots  & \ddots & \vdots\\
0 & v_n & 0 & \cdots & 0 
\end{pmatrix} : v \in \RR^n
\right\}.
\]
For 
\[
C = \cos(\varphi)C_0 + \sin(\varphi)JC_0
\]
and $u \in V(C_0)$ we have
\begin{align*}
& C(\cos(\varphi/2)u + \sin(\varphi/2)Ju) \\
& = \cos(\varphi/2)Cu + \sin(\varphi/2)CJu \\
& = \cos(\varphi/2)Cu - \sin(\varphi/2)JCu \\
& = \cos(\varphi/2)(\cos(\varphi)C_0 + \sin(\varphi)JC_0)u - \sin(\varphi/2)J(\cos(\varphi)C_0 + \sin(\varphi)JC_0)u \\
& = (\cos(\varphi/2)\cos(\varphi) + \sin(\varphi/2)\sin(\varphi))u + 
(\cos(\varphi/2)\sin(\varphi) - \sin(\varphi/2)\cos(\varphi))Ju \\
& = \cos(\varphi/2)u + \sin(\varphi/2)Ju.
\end{align*}
It follows that
\[
V(C) =  \{ \cos(\varphi/2)u + \sin(\varphi/2)Ju : u \in V(C_0) \}.
\]
Geometrically this tells us that, if we rotate a real structure by angle $\varphi$, then the $\pm 1$-eigenspaces rotate by angle $\varphi/2$.

The Riemannian metric $g$, the K\"{a}hler structure $J$ and a real structure $C$ on ${Q^n}^*$ can be used to give an explicit expression of the Riemannian curvature tensor $\bar{R}$ of $({Q^n}^*,g)$ (see \cite{Re96} and use duality). More precisely, we have
\begin{align*}
\bar{R}(X,Y)Z & =  g(X,Z)Y - g(Y,Z)X + g(JX,Z)JY -  g(JY,Z)JX  + 2g(JX,Y)JZ \\
 &  \qquad + g(CX,Z)CY - g(CY,Z)CX + g(JCX,Z)JCY  - g(JCY,Z)JCX 
\end{align*}
for all $X,Y,Z \in \cX({Q^n}^*)$, where $C$ is an arbitrary real structure in $\cA_0$. 
 
For every non-zero tangent vector $v \in \cp \cong T_o{Q^n}^*$ there exists a maximal abelian subspace $\ca \subset \cp$ with $v \in \ca$. If $\ca$ is unique, then $v$ is said to be a regular tangent vector, otherwise $v$ is said to be a singular tangent vector. From the explicit expression of the Riemannian curvature tensor it is straightforward to find the singular tangent vectors of ${Q^n}^*$. There are exactly two types of singular tangent vectors $v \in T_o{Q^n}^*$, which can be characterized as follows:
\begin{enumerate}
\item[(i)] If there exists a real structure $C \in \cA_0$ such that $v \in V(C)$, then $v$ is singular. Such a singular tangent vector is called $\cA$-principal.
\item[(ii)] If there exist a real structure $C \in \cA_0$ and orthonormal vectors $u,w \in V(C)$ such that $\frac{v}{||v||} = \frac{1}{\sqrt{2}}(u+Jw)$, then $v$ is singular. Such a singular tangent vector is called $\cA$-isotropic.
\end{enumerate}
For every unit tangent vector $v \in T_o{Q^n}^*$ there exist a real structure $C \in \cA_0$ and orthonormal vectors $u,w \in V(C)$ such that
\[
v = \cos(t)u + \sin(t)Jw
\]
for some $t \in [0,\frac{\pi}{4}]$. The singular tangent vectors correspond to the boundary values $t = 0$ and $t = \frac{\pi}{4}$.

Let $v$ be a unit tangent vector of ${Q^n}^*$ and consider the Jacobi operator $\bar{R}_v$ defined by
\[
\bar{R}_v X = \bar{R}(X,v)v.
\]
We have
\[
\bar{R}_v X  =  -X + g(X,v)v  - 3g(X,Jv)Jv + g(X,Cv)Cv - g(Cv,v)CX  + g(X,JCv)JCv .
\]
By a straightforward computation we obtain the eigenvalues and eigenspaces of $\bar{R}_v$ (see also \cite{Re96}). The eigenvalues are 
\[
0,-1+\cos(2t),-1-\cos(2t),-2+2\sin(2t),-2-2\sin(2t)
\] 
with corresponding eigenspaces
\begin{align*}
E_0 & = \RR u \oplus \RR w \cong \RR^2, \\
E_{-1+\cos(2t)} & = V(C) \ominus (\RR u \oplus \RR w) \cong \RR^{n-2}, \\
E_{-1-\cos(2t)} & = JV(C) \ominus J(\RR u \oplus \RR w) \cong \RR^{n-2} ,\\
E_{-2+2\sin(2t)} & = \RR(Ju+w) \cong \RR, \\
E_{-2-2\sin(2t)} & = \RR(Ju-w) \cong \RR ,
\end{align*}
where $C$ is a suitable real structure and $u,w \in V(C)$ are orthonormal vectors such that
\[
v = \cos(t)u + \sin(t)Jw
\]
for some $t \in [0,\frac{\pi}{4}]$. The five eigenvalues are distinct unless $t \in \{0,\tan^{-1}(\frac{1}{2}),\frac{\pi}{4}\}$.

If $t = 0$, then $Cv = v$ and hence $v$ is $\cA$-principal. In this case $\bar{R}_v$ has two eigenvalues $0,-2$ with corresponding eigenspaces
\begin{align*}
E_0 & = \RR v \oplus J(V(C)\ominus \RR v) \cong \RR^n,\\
E_{-2} & = \RR Jv \oplus (V(C)\ominus \RR v) \cong \RR^n.
\end{align*}

If $t = \frac{\pi}{4}$, then $v = \frac{1}{\sqrt{2}}(u+Jw)$ and hence $v$ is $\cA$-isotropic. In this case $\bar{R}_v$ has three eigenvalues $0,-1,-4$ with corresponding eigenspaces
\begin{align*}
E_0 & = \RR v \oplus \RR Cv \oplus \RR JCv = \RR v \oplus \CC Cv \cong \RR \oplus \CC,\\
E_{-1} & = \cp \ominus (\CC v \oplus \CC Cv) \cong \CC^{n-2}, \\
E_{-4} & = \RR Jv \cong \RR.
\end{align*}

If $t = \tan^{-1}(\frac{1}{2})$, then $\cos(t) = \frac{2}{\sqrt{5}}$, $\sin(t) = \frac{1}{\sqrt{5}}$, and hence $\cos(2t) = \frac{3}{5}$ and $\sin(2t) = \frac{4}{5}$. In this case $\bar{R}_v$ has four eigenvalues $0,-\frac{2}{5},-\frac{8}{5},-\frac{18}{5}$.

Let $\ca$ be a maximal abelian subspace of $\cp$ and $\ca^*$ be the dual vector space of $\ca$. For each $\alpha \in \ca^*$ we define 
\[
\cg_\alpha = \{ X \in \cg : \ad(H)X = \alpha(H)X \mbox{ for all } H \in \ca\}.
\] 
If $\alpha \neq 0$ and $\cg_\alpha \neq \{0\}$, then $\alpha$ is a restricted root and $\cg_\alpha$ is a restricted root space. Let $\Sigma \subset \ca^*$ be the set of restricted roots.  The restricted root spaces provide a restricted root space decomposition 
\[
\cg = \cg_0 \oplus \left( \bigoplus_{\alpha \in \Sigma} \cg_\alpha \right)
\]
of $\cg$, where $\cg_0 = \ck_0 \oplus \ca$ and $\ck_0 \cong \cs\co_{n-2}$ is the centralizer of $\ca$ in $\ck$. The restricted root spaces $\cg_\alpha$ and $\cg_0$ are pairwise orthogonal with respect to $B_\theta$. The corresponding restricted root system is of type $B_2$. We choose a set $\Lambda = \{\alpha_1,\alpha_2\}$ of simple roots of $\Sigma$ such that $\alpha_1$ is the longer root of the two simple roots, and denote by $\Sigma^+$ the resulting set of positive restricted roots. If we write, as usual, $\alpha_1 = \epsilon_1 - \epsilon_2$ and $\alpha_2 = \epsilon_2$, the positive restricted roots are 
\[
\alpha_1  = \epsilon_1 - \epsilon_2,\ 
\alpha_2  = \epsilon_2,\ 
\alpha_1 + \alpha_2  = \epsilon_1,\ 
\alpha_1 + 2\alpha_2  = \epsilon_1 + \epsilon_2.
\]
The multiplicities of the two long roots $\alpha_1$ and $\alpha_1 + 2\alpha_2$ are equal to $1$, and the multiplicities of the two short roots $\alpha_2$ and $\alpha_1 + \alpha_2$ are equal to $n-2$, respectively. Explicitly, the positive restricted root spaces and $\cg_0$ are:
\begin{eqnarray*}
\cg_0 & = &
 \left\{  
 \begin{pmatrix} 
0 & 0 & a_1 & 0 & 0 & \cdots & 0 \\
0 & 0 & 0 & a_2 & 0 & \cdots & 0 \\
a_1 & 0 & 0 & 0 & 0 & \cdots & 0 \\
0 & a_2 & 0 & 0 & 0 & \cdots & 0 \\
0 & 0 & 0 & 0 & & & \\
\vdots & \vdots & \vdots & \vdots & & B & \\
0 & 0 & 0 & 0 & & & 
\end{pmatrix}
:  a_1,a_2 \in \RR,\ B \in \cs\co_{n-2}
\right\} \cong \RR^2 \oplus \cs\co_{n-2},\\
\cg_{\alpha_1+\alpha_2} & = & \left\{ 
\begin{pmatrix} 
0 & 0 & 0 & 0 & v_1 & \cdots & v_{n-2} \\
0 & 0 & 0 & 0 & 0 & \cdots & 0 \\
0 & 0 & 0 & 0 & v_1 & \cdots & v_{n-2} \\
0 & 0 & 0 & 0 & 0 & \cdots & 0 \\
v_1 & 0 & -v_1 & 0 & 0 & \cdots & 0 \\
\vdots & \vdots & \vdots & \vdots & \vdots & \ddots & \vdots\\
v_{n-2} & 0 & -v_{n-2} & 0 & 0 & \cdots & 0
\end{pmatrix}
:  v \in \RR^{n-2} \right\} \cong \RR^{n-2}, \\
\cg_{\alpha_2} & = & \left\{ 
\begin{pmatrix} 
0 & 0 & 0 & 0 & 0 & \cdots & 0 \\
0 & 0 & 0 & 0 & w_1 & \cdots & w_{n-2} \\
0 & 0 & 0 & 0 & 0 & \cdots & 0 \\
0 & 0 & 0 & 0 &  w_1 & \cdots & w_{n-2} \\
0 & w_1 & 0 & -w_1 & 0 & \cdots & 0 \\
\vdots & \vdots & \vdots & \vdots & \vdots & \ddots & \vdots\\
0 & w_{n-2} & 0 & -w_{n-2} & 0 & \cdots & 0
\end{pmatrix}
:  w \in \RR^{n-2} \right\} \cong \RR^{n-2}, \\
\cg_{\alpha_1} & = & \left\{ 
\begin{pmatrix} 
0 & x & 0 & x & 0 & \cdots & 0 \\
-x & 0 & x & 0 & 0 & \cdots & 0 \\
0 & x & 0 & x & 0 & \cdots & 0 \\
x & 0 & -x & 0 &  0 & \cdots & 0 \\
0 & 0 & 0 & 0 & 0 & \cdots & 0 \\
\vdots & \vdots & \vdots & \vdots & \vdots & \ddots & \vdots\\
0 & 0 & 0 & 0 & 0 & \cdots & 0
\end{pmatrix}
:  x \in \RR \right\} \cong \RR ,\\
\cg_{\alpha_1+2\alpha_2} & = & \left\{  
\begin{pmatrix} 
0 & y & 0 & -y & 0 & \cdots & 0 \\
-y & 0 & y & 0 & 0 & \cdots & 0 \\
0 & y & 0 & -y & 0 & \cdots & 0 \\
-y & 0 & y & 0 &  0 & \cdots & 0 \\
0 & 0 & 0 & 0 & 0 & \cdots & 0 \\
\vdots & \vdots & \vdots & \vdots & \vdots & \ddots & \vdots\\
0 & 0 & 0 & 0 & 0 & \cdots & 0
\end{pmatrix}
:  y \in \RR \right\} \cong \RR.
\end{eqnarray*}
The negative restricted root spaces can be computed easily from the positive restricted root spaces using the fact that $\cg_{-\alpha} = \theta(\cg_\alpha)$.

For each $\alpha \in \Sigma$ we define
\begin{equation*}
\ck_\alpha = \ck \cap (\cg_\alpha \oplus \cg_{-\alpha}) ,\ 
\cp_\alpha = \cp \cap (\cg_\alpha \oplus \cg_{-\alpha}).
\end{equation*}
Then we have $\cp_\alpha = \cp_{-\alpha}$,
$\ck_\alpha = \ck_{-\alpha}$ and $\cp_\alpha \oplus \ck_\alpha = \cg_\alpha
\oplus \cg_{-\alpha}$ for all $\alpha \in \Sigma$. 

We define a nilpotent subalgebra $\cn$ of $\cg$ by 
\begin{align*}
\cn & = \cg_{\alpha_1} \oplus \cg_{\alpha_2} \oplus \cg_{\alpha_1 + \alpha_2} \oplus \cg_{\alpha_1 + 2\alpha_2} \\
& = \left\{ 
\begin{pmatrix}
0 & x+y & 0 & x-y & v_1 & \cdots & v_{n-2} \\
-x-y & 0 & x+y & 0 & w_1 & \cdots & w_{n-2} \\
0 & x+y & 0 & x-y & v_1 & \cdots & v_{n-2} \\
x-y & 0 & -x+y & 0 & w_1 & \cdots & w_{n-2} \\
v_1 & w_1 & -v_1 & -w_1 & 0 & \cdots & 0 \\
\vdots & \vdots & \vdots & \vdots & \vdots & \ddots & \vdots\\
v_{n-2} & w_{n-2} & -v_{n-2} & -w_{n-2} & 0 & \cdots & 0
\end{pmatrix} : 
\begin{array}{l} x,y \in \RR,\\ v,w \in \RR^{n-2} \end{array}
\right\}.
\end{align*}
Then $\cg = \ck \oplus \ca \oplus \cn$ is an Iwasawa decomposition of $\cg$, which induces a corresponding Iwasawa decomposition $G = KAN$ of $G$. Here, $A$ and $N$ are the connected closed subgroups of $G$ with Lie algebras $\ca$ and $\cn$, respectively. 

The subalgebra 
\[
\ca \oplus \cn = 
\left\{ 
\begin{pmatrix}
0 & x+y & a_1 & x-y & v_1 & \cdots & v_{n-2} \\
-x-y & 0 & x+y & a_2 & w_1 & \cdots & w_{n-2} \\
a_1 & x+y & 0 & x-y & v_1 & \cdots & v_{n-2} \\
x-y & a_2 & -x+y & 0 & w_1 & \cdots & w_{n-2} \\
v_1 & w_1 & -v_1 & -w_1 & 0 & \cdots & 0 \\
\vdots & \vdots & \vdots & \vdots & \vdots & \ddots & \vdots\\
v_{n-2} & w_{n-2} & -v_{n-2} & -w_{n-2} & 0 & \cdots & 0
\end{pmatrix} : 
\begin{array}{l} a_1,a_2,x,y \in \RR,\\ 
v,w \in \RR^{n-2} \end{array}
\right\}
\]
 of $\cg$ is solvable and the corresponding connected closed subgroup $AN$ of $G$ with Lie algebra $\ca \oplus \cn$ is solvable, simply connected, and acts simply transitively on ${Q^n}^*$. Then $({Q^n}^*,g)$ is isometric to the solvable Lie group $AN$ equipped with the left-invariant Riemannian metric $\langle \cdot , \cdot \rangle$ defined by
\begin{align*}
\langle H_1 + \hat{X}_1 , H_2 + \hat{X}_2 \rangle & = -\frac{1}{4n}B(H_1,\theta(H_2)) - \frac{1}{8n}B(\hat{X}_1,\theta(\hat{X}_2)) \\
& = -\frac{1}{4}\tr(H_1\theta(H_2)) - \frac{1}{8}\tr(\hat{X}_1\theta(\hat{X}_2)) \\
& = \frac{1}{4}\tr(H_1H_2) - \frac{1}{8}\tr(\hat{X}_1\theta(\hat{X}_2))
\end{align*}
with $H_1,H_2 \in \ca$ and $\hat{X}_1,\hat{X}_2 \in \cn$. For each $\hat{X} \in \cn$, the orthogonal projection $X$ onto $\cp$ with respect to $B_\theta$ is
\[
X = \frac{1}{2}(\hat{X} - \theta(\hat{X})) \in \cp.
\]
By construction, we have $\langle \hat{X} , \hat{X} \rangle = g(X,X)$ and 
\[
\langle H_1 + \hat{X}_1 , H_2 + \hat{X}_2 \rangle = g(H_1 + X_1,H_2 + X_2).
\]

Let $H^1,H^2 \in \ca$ be the dual basis of $\alpha_1,\alpha_2 \in \ca^\ast$ defined by $\alpha_\nu(H^\mu) = \delta_{\nu\mu}$.
Since $\alpha_1 = \epsilon_1 - \epsilon_2$ and $\alpha_2 = \epsilon_2$, we have
\[
H^1 = \begin{pmatrix} 
0 & 0 & 1 & 0 & 0 & \cdots & 0 \\
0 & 0 & 0 & 0 & 0 & \cdots & 0 \\
1 & 0 & 0 & 0 & 0 & \cdots & 0 \\
0 & 0 & 0 & 0 & 0 & \cdots & 0 \\
0 & 0 & 0 & 0 & 0 & \cdots & 0 \\
\vdots & \vdots & \vdots & \vdots & \vdots & \ddots & \vdots\\
0 & 0 & 0 & 0 & 0 & \cdots & 0 \\
\end{pmatrix}
,\ H^2 = \begin{pmatrix} 
0 & 0 & 1 & 0 & 0 & \cdots & 0 \\
0 & 0 & 0 & 1 & 0 & \cdots & 0 \\
1 & 0 & 0 & 0 & 0 & \cdots & 0 \\
0 & 1 & 0 & 0 & 0 & \cdots & 0 \\
0 & 0 & 0 & 0 & 0 & \cdots & 0 \\
\vdots & \vdots & \vdots & \vdots & \vdots & \ddots & \vdots\\
0 & 0 & 0 & 0 & 0 & \cdots & 0 \\
\end{pmatrix}.
\]
Note that
\[
\langle H^1 , H^1 \rangle = \frac{1}{4}\tr(H^1H^1) = \frac{1}{2},\ 
\langle H^2 , H^2 \rangle = \frac{1}{4}\tr(H^2H^2) = 1.
\]

For each $\alpha$ in $\Sigma$ we define the root vector $H_\alpha \in \ca$ of $\alpha$ by
$\langle H_\alpha , H \rangle = \alpha(H)$
for all $H \in \ca$.
Note that
\[
[H,X_\alpha] = \ad(H)X_\alpha = \alpha(H)X_\alpha = \langle H_\alpha , H \rangle X_\alpha
\]  
for all $H \in \ca$ and $X_\alpha \in \cg_\alpha$.
If we put
\[
H_{\alpha} = \begin{pmatrix} 
0 & 0 & x_1 & 0 & 0 & \cdots & 0 \\
0 & 0 & 0 & x_2 & 0 & \cdots & 0 \\
x_1 & 0 & 0 & 0 & 0 & \cdots & 0 \\
0 & x_2 & 0 & 0 & 0 & \cdots & 0 \\
0 & 0 & 0 & 0 & 0 & \cdots & 0 \\
\vdots & \vdots & \vdots & \vdots & \vdots & \ddots & \vdots\\
0 & 0 & 0 & 0 & 0 & \cdots & 0 \\
\end{pmatrix}
,\ 
H = \begin{pmatrix} 
0 & 0 & a_1 & 0 & 0 & \cdots & 0 \\
0 & 0 & 0 & a_2 & 0 & \cdots & 0 \\
a_1 & 0 & 0 & 0 & 0 & \cdots & 0 \\
0 & a_2 & 0 & 0 & 0 & \cdots & 0 \\
0 & 0 & 0 & 0 & 0 & \cdots & 0 \\
\vdots & \vdots & \vdots & \vdots & \vdots & \ddots & \vdots\\
0 & 0 & 0 & 0 & 0 & \cdots & 0 \\
\end{pmatrix},
\]
then
\[
\langle H_\alpha , H \rangle = \frac{1}{4}\tr(H_\alpha H) = \frac{1}{2}(x_1a_1 + x_2a_2).
\]
It follows that
\[
\begin{matrix}
H_{\alpha_1} & = & \begin{pmatrix} 
0 & 0 & 2 & 0 & 0 & \cdots & 0 \\
0 & 0 & 0 & -2 & 0 & \cdots & 0 \\
2 & 0 & 0 & 0 & 0 & \cdots & 0 \\
0 & -2 & 0 & 0 & 0 & \cdots & 0 \\
0 & 0 & 0 & 0 & 0 & \cdots & 0 \\
\vdots & \vdots & \vdots & \vdots & \vdots & \ddots & \vdots\\
0 & 0 & 0 & 0 & 0 & \cdots & 0 
\end{pmatrix} & , & H_{\alpha_1+2\alpha_2} & = & \begin{pmatrix} 
0 & 0 & 2 & 0 & 0 & \cdots & 0 \\
0 & 0 & 0 & 2 & 0 & \cdots & 0 \\
2 & 0 & 0 & 0 & 0 & \cdots & 0 \\
0 & 2 & 0 & 0 & 0 & \cdots & 0 \\
0 & 0 & 0 & 0 & 0 & \cdots & 0 \\
\vdots & \vdots & \vdots & \vdots & \vdots & \ddots & \vdots\\
0 & 0 & 0 & 0 & 0 & \cdots & 0 
\end{pmatrix} & ,\\ 
& & & & & & & \\
H_{\alpha_2} & = &  \begin{pmatrix} 
0 & 0 & 0 & 0 & 0 & \cdots & 0 \\
0 & 0 & 0 & 2 & 0 & \cdots & 0 \\
0 & 0 & 0 & 0 & 0 & \cdots & 0 \\
0 & 2 & 0 & 0 & 0 & \cdots & 0 \\
0 & 0 & 0 & 0 & 0 & \cdots & 0 \\
\vdots & \vdots & \vdots & \vdots & \vdots & \ddots & \vdots\\
0 & 0 & 0 & 0 & 0 & \cdots & 0 
\end{pmatrix} & , &
H_{\alpha_1+\alpha_2} & = & \begin{pmatrix} 
0 & 0 & 2 & 0 & 0 & \cdots & 0 \\
0 & 0 & 0 & 0 & 0 & \cdots & 0 \\
2 & 0 & 0 & 0 & 0 & \cdots & 0 \\
0 & 0 & 0 & 0 & 0 & \cdots & 0 \\
0 & 0 & 0 & 0 & 0 & \cdots & 0 \\
\vdots & \vdots & \vdots & \vdots & \vdots & \ddots & \vdots\\
0 & 0 & 0 & 0 & 0 & \cdots & 0 
\end{pmatrix} & . 
\end{matrix}
\]
We have
\begin{align*}
\langle H_{\alpha_1} , H_{\alpha_1} \rangle & = \frac{1}{4}\tr(H_{\alpha_1}H_{\alpha_1}) = 4,\\
\langle H_{\alpha_1+2\alpha_2} , H_{\alpha_1+2\alpha_2} \rangle & = \frac{1}{4}\tr(H_{\alpha_1+2\alpha_2}H_{\alpha_1+2\alpha_2}) = 4,\\
\langle H_{\alpha_2} , H_{\alpha_2} \rangle & = \frac{1}{4}\tr(H_{\alpha_2}H_{\alpha_2}) = 2,\\
\langle H_{\alpha_1+\alpha_2}  , H_{\alpha_1+\alpha_2}  \rangle & = \frac{1}{4}\tr(H_{\alpha_1+\alpha_2} H_{\alpha_1+\alpha_2} ) = 2,
\end{align*}
and
\[
2H^1 = H_{\alpha_1+\alpha_2} \qquad \mbox{and} \qquad
2H^2 = H_{\alpha_1+2\alpha_2}.
\]

\section{The homogeneous complex hypersurface}
\label{hcs}

In this section we construct a homogeneous complex hypersurface $P^{n-1} \cong \CC H^{n-1}(-4)$ in $({Q^n}^*,g)$ and compute its shape operator. We define
\[
\ch^{2n-3} = \cg_{\alpha_2} \oplus \cg_{\alpha_1+\alpha_2} \oplus \cg_{\alpha_1+2\alpha_2}.
\]
It is easy to verify that $\ch^{2n-3}$ is a nilpotent subalgebra of $\cn$ and isomorphic to the $(2n-3)$-dimensional Heisenberg algebra with $1$-dimensional center.

We have
\[
[H^2,\hat{X}]  = \begin{cases}
\hat{X} &, \mbox{ if } \hat{X} \in \cg_{\alpha_2} \oplus \cg_{\alpha_1+\alpha_2}, \\
2\hat{X} &, \mbox{ if } \hat{X} \in \cg_{\alpha_1+2\alpha_2}.
\end{cases}
\]
It follows that 
\[
\cd = \RR H^2 \oplus \ch^{2n-3} = \RR H_{\alpha_1+2\alpha_2} \oplus \cg_{\alpha_2} \oplus \cg_{\alpha_1+\alpha_2} \oplus \cg_{\alpha_1+2\alpha_2}
\] 
is a solvable subalgebra of $\ca \oplus \cn$. (Note that $\RR H^2$ denotes here the real span of $H^2$ and not the real hyperbolic plane!) In fact, this subalgebra is the standard solvable extension of the Heisenberg algebra $\ch^{2n-3}$ and isomorphic to the solvable Lie algebra of the solvable part of the Iwasawa decomposition of the isometry group of the complex hyperbolic space $\CC H^{n-1}(-4)$ (see e.g.\ \cite{BTV95} or \cite{Ta08}). 

This construction leads to an isometric embedding $\hat{P}^{n-1}$ of the $(n-1)$-dimensional complex hyperbolic space $\CC H^{n-1}(-4)$ with constant holomorphic sectional curvature $-4$ into $(AN,\langle \cdot , \cdot \rangle)$. By construction, $\hat{P}^{n-1}$ is a homogeneous submanifold of $(AN,\langle \cdot , \cdot \rangle)$. Let $\hat{J}$ be the complex structure on $(AN,\langle \cdot , \cdot \rangle)$ corresponding to the complex structure $J$ on $({Q^n}^*,g)$. 
We have $\hat{J}\cg_{\alpha_2} = \cg_{\alpha_1+\alpha_2}$ and $\hat{J}H^2 \in \cg_{\alpha_1+2\alpha_2}$, which shows that the tangent space 
\[
T_o\hat{P}^{n-1} = \RR H_{\alpha_1+2\alpha_2} \oplus \cg_{\alpha_2} \oplus \cg_{\alpha_1+\alpha_2} \oplus \cg_{\alpha_1+2\alpha_2} 
\] 
is a complex subspace of $T_oAN$. Since $AN$ is contained in the identity component $SO^o_{2,n}$ of the full isometry group of ${Q^n}^*$, it consists of holomorphic isometries, which implies that $\hat{P}^{n-1}$ is a complex submanifold of $(AN,\langle \cdot , \cdot \rangle)$.

Altogether we conclude that the solvable subalgebra
\[
\cd = \RR H_{\alpha_1+2\alpha_2} \oplus \cg_{\alpha_2} \oplus \cg_{\alpha_1+\alpha_2} \oplus \cg_{\alpha_1+2\alpha_2}
\]
of $\ca \oplus \cn$ induces an isometric embedding $\hat{P}^{n-1}$ of the complex hyperbolic space $\CC H^{n-1}(-4)$ with constant holomorphic sectional curvature $-4$ into $(AN,\langle \cdot , \cdot \rangle)$ as a homogeneous complex hypersurface. This induces an isometric embedding $P^{n-1}$ of the complex hyperbolic space $\CC H^{n-1}(-4)$ with constant holomorphic sectional curvature $-4$ into $({Q^n}^*,g)$ as a homogeneous complex hypersurface.

\begin{re}
\rm Smyth \cite{Sm68} proved that every homogeneous complex hypersurface in the complex hyperbolic space $\CC H^n$ is a complex hyperbolic hyperplane $\CC H^{n-1}$ embedded in $\CC H^n$ as a totally geodesic submanifold. As we have just seen, up to congruency, there are at least two homogeneous complex hypersurfaces in the complex hyperbolic quadric ${Q^n}^*$, namely the complex hyperbolic quadric ${Q^{n-1}}^*$ and the complex hyperbolic space $P^{n-1} \cong \CC H^{n-1}(-4)$. The first one is totally geodesic (see \cite{CN77} or \cite{Kl08} and use duality between Riemannian symmetric spaces of compact type and of non-compact type), the second one is not. The classification of the homogeneous complex hypersurfaces in the complex hyperbolic quadric ${Q^n}^*$ remains an open problem.
\end{re}

\medskip
We now compute the shape operator $\hat{A}$ of $\hat{P}^{n-1} \cong \CC H^{n-1}(-4)$ in $(AN,\langle \cdot , \cdot \rangle)$. Let 
\[
\hat\zeta \in (\ca \ominus \RR H^2) \oplus \cg_{\alpha_1} = \RR H_{\alpha_1} \oplus \cg_{\alpha_1}
\] 
be a unit normal vector of $\hat{P}^{n-1}$ at $o$. The Weingarten equation tells us that
\[
\langle \hat{A}_{\hat{\zeta}} \hat{X} , \hat{Y} \rangle = - \langle \hat\nabla_{\hat{X}}\hat{\zeta} , \hat{Y} \rangle,
\]
where $\hat\nabla$ is the Levi Civita covariant derivative of $(AN,\langle \cdot , \cdot \rangle)$ and $\hat{X},\hat{Y} \in \cd$. We consider $\hat{\zeta},\hat{X},\hat{Y}$ as left-invariant vector fields. Since $\langle \cdot , \cdot \rangle$ is a left-invariant Riemannian metric, the Koszul formula for $\hat\nabla$ implies
\[
2\langle \hat{A}_{\hat{\zeta}}\hat{X} , \hat{Y} \rangle = 2\langle \hat\nabla_{\hat{X}} \hat{Y}, \hat{\zeta} \rangle = \langle [\hat{X},\hat{Y}],\hat{\zeta} \rangle + \langle [\hat{\zeta},\hat{X}],\hat{Y} \rangle  + \langle [\hat{\zeta},\hat{Y}],\hat{X} \rangle  .
\]
Since $\cd$ is a subalgebra of $\ca \oplus \cn$, we have $[\hat{X},\hat{Y}] \in \cd$ and hence $\langle [\hat{X},\hat{Y}],\hat{\zeta} \rangle = 0$. Moreover, since
$\ad(\hat{\zeta})^* =  -\ad(\theta(\hat{\zeta}))$, we have
\[
\langle [\hat{\zeta},\hat{Y}],\hat{X} \rangle = -\langle [\theta(\hat{\zeta}),\hat{X}],\hat{Y} \rangle.
\]
Altogether this implies
\[
2\langle \hat{A}_{\hat{\zeta}} \hat{X} , \hat{Y} \rangle = \langle [\hat{\zeta}-\theta(\hat{\zeta}),\hat{X}],\hat{Y} \rangle.
\]
Thus, the shape operater $\hat{A}_{\hat{\zeta}}$ of $\hat{P}^{n-1}$ is given by
\[
\hat{A}_{\hat{\zeta}} \hat{X} = [\zeta,\hat{X}]_{\cd},
\]
where
\[
\zeta = \frac{1}{2}(\hat\zeta - \theta(\hat\zeta)) \in \cp
\] 
is the orthogonal projection of $\hat{\zeta}$ onto $\cp$ and $[\, \cdot\ ]_\cd$ is the orthogonal projection onto $\cd$.

The normal space $\nu_o\hat{P}^{n-1}$ of $\hat{P}^{n-1}$ at the point $o$ is given by
\[
\nu_o\hat{P}^{n-1} = \RR H_{\alpha_1} \oplus \cg_{\alpha_1} 
= \left\{ \begin{pmatrix} 
0 & x & a & x & 0 & \cdots & 0 \\
-x & 0 & x & -a & 0 & \cdots & 0 \\
a & x & 0 & x & 0 & \cdots & 0 \\
x & -a & -x & 0 &  0 & \cdots & 0 \\
0 & 0 & 0 & 0 & 0 & \cdots & 0 \\
\vdots & \vdots & \vdots & \vdots & \vdots & \ddots & \vdots\\
0 & 0 & 0 & 0 & 0 & \cdots & 0
\end{pmatrix}
:  a,x \in \RR \right\} ,
\]
and the tangent space $T_o\hat{P}^{n-1}$ of $\hat{P}^{n-1}$ at the point $o$ is given by
\[
T_o\hat{P}^{n-1} = \cd = \left\{ 
\begin{pmatrix}
0 & y & b & -y & v_1 & \cdots & v_{n-2} \\
-y & 0 & y & b & w_1 & \cdots & w_{n-2} \\
b & y & 0 & -y & v_1 & \cdots & v_{n-2} \\
-y & b & y & 0 & w_1 & \cdots & w_{n-2} \\
v_1 & w_1 & -v_1 & -w_1 & 0 & \cdots & 0 \\
\vdots & \vdots & \vdots & \vdots & \vdots & \ddots & \vdots\\
v_{n-2} & w_{n-2} & -v_{n-2} & -w_{n-2} & 0 & \cdots & 0
\end{pmatrix} :
\begin{array}{l} b,y \in \RR,\\ 
v,w \in \RR^{n-2} \end{array} \right\} .
\]

The vector $\hat\zeta = \frac{1}{2}H_{\alpha_1} \in \ca$ is a unit normal vector of $\hat{P}^{n-1}$ at $o$. We have 
\[
\theta(\hat\zeta) =  \frac{1}{2}\theta(H_{\alpha_1}) = -\frac{1}{2}H_{\alpha_1} = - \hat\zeta
\] 
and thus
\[
\zeta = \frac{1}{2}(\hat\zeta - \theta(\hat\zeta)) = \hat\zeta.
\]

A straightforward matrix computation gives
\begin{align*}
& \left[  \left( \begin{smallmatrix} 
0 & 0 & 1 & 0 & 0 & \cdots & 0 \\
0 & 0 & 0 & -1 & 0 & \cdots & 0 \\
1 & 0 & 0 & 0 & 0 & \cdots & 0 \\
0 & -1 & 0 & 0 &  0 & \cdots & 0 \\
0 & 0 & 0 & 0 & 0 & \cdots & 0 \\
\vdots & \vdots & \vdots & \vdots & \vdots & \ddots & \vdots\\
0 & 0 & 0 & 0 & 0 & \cdots & 0
\end{smallmatrix} \right), 
\left( \begin{smallmatrix}
0 & y & b & -y & v_1 & \cdots & v_{n-2} \\
-y & 0 & y & b & w_1 & \cdots & w_{n-2} \\
b & y & 0 & -y & v_1 & \cdots & v_{n-2} \\
-y & b & y & 0 & w_1 & \cdots & w_{n-2} \\
v_1 & w_1 & -v_1 & -w_1 & 0 & \cdots & 0 \\
\vdots & \vdots & \vdots & \vdots & \vdots & \ddots & \vdots\\
v_{n-2} & w_{n-2} & -v_{n-2} & -w_{n-2} & 0 & \cdots & 0
\end{smallmatrix} \right)  \right] \\ & = 
\left( \begin{smallmatrix}
0 & 0 & 0 & 0 & v_1 & \cdots & v_{n-2} \\
0 & 0 & 0 & 0 & -w_1 & \cdots & -w_{n-2} \\
0 & 0 & 0 & 0 & v_1 & \cdots & v_{n-2} \\
0 & 0 & 0 & 0 & -w_1 & \cdots & -w_{n-2} \\
v_1 & -w_1 & -v_1 & w_1 & 0 & \cdots & 0 \\
\vdots & \vdots & \vdots & \vdots & \vdots & \ddots & \vdots\\
v_{n-2} & -w_{n-2} & -v_{n-2} & w_{n-2} & 0 & \cdots & 0
\end{smallmatrix} \right) \in \cd.
\end{align*}
Since the latter matrix is in $\cd$, we conclude that
\[
\hat{A}_{\hat{\zeta}} \hat{X} = \begin{pmatrix}
0 & 0 & 0 & 0 & v_1 & \cdots & v_{n-2} \\
0 & 0 & 0 & 0 & -w_1 & \cdots & -w_{n-2} \\
0 & 0 & 0 & 0 & v_1 & \cdots & v_{n-2} \\
0 & 0 & 0 & 0 & -w_1 & \cdots & -w_{n-2} \\
v_1 & -w_1 & -v_1 & w_1 & 0 & \cdots & 0 \\
\vdots & \vdots & \vdots & \vdots & \vdots & \ddots & \vdots\\
v_{n-2} & -w_{n-2} & -v_{n-2} & w_{n-2} & 0 & \cdots & 0
\end{pmatrix}
\]
with 
\[
\hat{X} = \begin{pmatrix}
0 & y & b & -y & v_1 & \cdots & v_{n-2} \\
-y & 0 & y & b & w_1 & \cdots & w_{n-2} \\
b & y & 0 & -y & v_1 & \cdots & v_{n-2} \\
-y & b & y & 0 & w_1 & \cdots & w_{n-2} \\
v_1 & w_1 & -v_1 & -w_1 & 0 & \cdots & 0 \\
\vdots & \vdots & \vdots & \vdots & \vdots & \ddots & \vdots\\
v_{n-2} & w_{n-2} & -v_{n-2} & -w_{n-2} & 0 & \cdots & 0
\end{pmatrix} \in T_o\hat{P}^{n-1}.
\]
It follows that the principal curvatures of $\hat{P}^{n-1}$ with respect to the unit normal vector $\hat\zeta$ are $0$, $1$ and $-1$, with corresponding principal curvature spaces
\[
\hat{T}^{\hat\zeta}_0 = \RR H_{\alpha_1+2\alpha_2} \oplus \cg_{\alpha_1+2\alpha_2}\ ,\ \hat{T}^{\hat\zeta}_1 = \cg_{\alpha_1+\alpha_2}\ ,\ \hat{T}^{\hat\zeta}_{-1} = \cg_{\alpha_2}.
\]

We now compute the shape operator of $\hat{P}^{n-1}$ at $o$ for other unit normal vectors.
Since $\nu_o\hat{P}^{n-1}$ is $\hat{J}$-invariant, the vector $\hat{J}\hat\zeta \in \cg_{\alpha_1}$ is a unit normal vector of $\hat{P}^{n-1}$ at $o$. Moreover, $\hat\zeta,\hat{J}\hat\zeta$ is an orthonormal basis of the normal space $\nu_o\hat{P}^{n-1}$. Using a well-known formula for the shape operator of a complex submanifold of a K\"{a}hler manifold (see e.g.\ \cite{CR15}, Lemma 7.4), we have
\[
\hat{A}_{\hat{J}\hat\zeta} = \hat{J}\hat{A}_{\hat\zeta}.
\]
Since every unit normal vector of $\hat{P}^{n-1}$ at $o$ is of the form
\[
 \cos(\varphi)\hat\zeta + \sin(\varphi)\hat{J}\hat\zeta,
\]
the shape operator $\hat{A}_{\hat{\zeta}}$ therefore completely determines the shape operator for every other unit normal vector of $\hat{P}^{n-1}$ at $o$. More precisely, we have
\[
\hat{A}_{\cos(\varphi)\hat\zeta + \sin(\varphi)\hat{J}\hat\zeta} = 
\cos(\varphi)\hat{A}_{\hat\zeta} + \sin(\varphi)J\hat{A}_{\hat\zeta}.
\]
This readily implies that the principal curvatures of $\hat{P}^{n-1}$ with respect to the unit normal vector $\cos(\varphi)\hat\zeta + \sin(\varphi)\hat{J}\hat\zeta$ are $0$, $1$ and $-1$, with corresponding principal curvature spaces
\begin{align*}
\hat{T}^{\cos(\varphi)\hat\zeta + \sin(\varphi)\hat{J}\hat\zeta}_0 & 
= \RR H_{\alpha_1+2\alpha_2} \oplus \cg_{\alpha_1+2\alpha_2},\\
\hat{T}^{\cos(\varphi)\hat\zeta + \sin(\varphi)\hat{J}\hat\zeta}_1 & = 
\{  \cos\left(\textstyle{\frac{\varphi}{2}}\right)\hat{X} + \sin\left(\textstyle{\frac{\varphi}{2}}\right)J\hat{X} : \hat{X} \in \cg_{\alpha_1+\alpha_2}\},\\
\hat{T}^{\cos(\varphi)\hat\zeta + \sin(\varphi)\hat{J}\hat\zeta}_{-1} & 
= \{\sin\left(\textstyle{\frac{\varphi}{2}}\right)\hat{X} -  \cos\left(\textstyle{\frac{\varphi}{2}}\right)J\hat{X}  : \hat{X} \in \cg_{\alpha_1+\alpha_2}\}.
\end{align*}
Using orthogonal projections onto $\cp$ we obtain the corresponding description of the shape operator $A$ of $P^{n-1}$ at $o$.

Recall that 
\[
\nu_o\hat{P}^{n-1} = \RR H_{\alpha_1} \oplus \cg_{\alpha_1} .
\]
The orthogonal projection of $\nu_o\hat{P}^{n-1}$ onto $\cp$ is
\[
\nu_oP^{n-1} = \CC H_{\alpha_1} = \RR H_{\alpha_1} \oplus \cp_{\alpha_1} =
\left\{ 
\begin{pmatrix} 
0 & 0 & a & x & 0 & \cdots & 0 \\
0 & 0 & x & -a & 0 & \cdots & 0 \\
a & x & 0 & 0 & 0 & \cdots & 0 \\
x & -a & 0 & 0 &  0 & \cdots & 0 \\
0 & 0 & 0 & 0 & 0 & \cdots & 0 \\
\vdots & \vdots & \vdots & \vdots & \vdots & \ddots & \vdots\\
0 & 0 & 0 & 0 & 0 & \cdots & 0
\end{pmatrix}
:  a,x \in \RR \right\} .
\]
The complex line $\CC H_{\alpha_1}$ is a Lie triple system in $\cp$ and therefore determines a totally geodesic complex submanifold $B_1$ of ${Q^n}^*$. The (non-zero) tangent vectors of $B_1$ are $\cA$-isotropic, which implies that the sectional curvature of $B_1$ is equal to $-4$. Thus $B_1$ is isometric to the complex hyperbolic line $\CC H^1(-4)$ of constant (holomorphic) sectional curvature $-4$. We will encounter $B_1$ again later, where it appears in a horospherical decomposition of the complex hyperbolic quadric. 

We now apply the standard real structure $C_0$ to the normal space $\nu_oP^{n-1} = T_oB_1$,
\begin{align*}
C_0(T_oB_1) & = 
\left\{ 
\begin{pmatrix} 
0 & 0 & a & x & 0 & \cdots & 0 \\
0 & 0 & -x & a & 0 & \cdots & 0 \\
a & -x & 0 & 0 & 0 & \cdots & 0 \\
x & a & 0 & 0 &  0 & \cdots & 0 \\
0 & 0 & 0 & 0 & 0 & \cdots & 0 \\
\vdots & \vdots & \vdots & \vdots & \vdots & \ddots & \vdots\\
0 & 0 & 0 & 0 & 0 & \cdots & 0
\end{pmatrix}
:  a,x \in \RR \right\}  \\
& = \RR H_{\alpha_1+2\alpha_2} \oplus \cp_{\alpha_1+2\alpha_2} = \CC H_{\alpha_1+2\alpha_2} .
\end{align*}
Note that $C_0(T_oB_1) = C(T_oB_1)$ for any real structure $C$ at $o$ and therefore the construction is independent of the choice of real structure.
The complex line $\CC H_{\alpha_1+2\alpha_2}$ is also a Lie triple system in $\cp$ and determines a totally geodesic complex submanifold $\Sigma_1$ of ${Q^n}^*$. The (non-zero) tangent vectors of $\Sigma_1$ are also $\cA$-isotropic, which implies that the sectional curvature of $\Sigma_1$ is equal to $-4$. Thus $\Sigma_1$ is isometric to the complex hyperbolic line $\CC H^1(-4)$ of constant (holomorphic) sectional curvature $-4$. The tangent space $T_o\Sigma_1$ is the kernel of the shape operator of the homogeneous complex hypersurface $P^{n-1} \cong \CC H^{n-1}(-4)$. Since $\RR H_{\alpha_1+2\alpha_2} \oplus \cg_{\alpha_1+2\alpha_2}$ is a subalgebra of $\cd$, this implies geometrically that $\hat{P}^{n-1}$, and hence also $P^{n-1}$, is foliated by totally geodesic complex hyperbolic lines $\CC H^1(-4)$ whose tangent spaces are obtained by rotating the normal spaces of $\hat{P}^{n-1}$ (resp.\ $P^{n-1}$) via a real structure $\hat{C}$ (resp.\ $C$). 

The Riemannian product $B_1 \times \Sigma_1 \cong \CC H^1(-4) \times \CC H^1(-4)$ is isometric to the complex hyperbolic quadric ${Q^2}^*$ and describes the standard isometric embedding of ${Q^2}^*$ into ${Q^n}^*$.

We have
\[
[\cp_{\alpha_1+\alpha_2} \oplus \cp_{\alpha_2},\cp_{\alpha_1+\alpha_2} \oplus \cp_{\alpha_2}] \subset \ck_0 \oplus \ck_{\alpha_1} \oplus \ck_{\alpha_1+2\alpha_2}
\]
and
\begin{align*}
[\ck_0,\cp_{\alpha_1+\alpha_2}] & \subset \cp_{\alpha_1+\alpha_2},\\
[\ck_{\alpha_1},\cp_{\alpha_1+\alpha_2}] & \subset \cp_{\alpha_2},\\
[\ck_{\alpha_1+2\alpha_2},\cp_{\alpha_1+\alpha_2}] & \subset \cp_{\alpha_2},\\
[\ck_0,\cp_{\alpha_2}] & \subset \cp_{\alpha_2},\\
[\ck_{\alpha_1},\cp_{\alpha_2}] & \subset \cp_{\alpha_1+\alpha_2},\\
[\ck_{\alpha_1+2\alpha_2},\cp_{\alpha_2}] & \subset \cp_{\alpha_1+\alpha_2}.
\end{align*}
Altogether we conclude that $\cp_{\alpha_1+\alpha_2} \oplus \cp_{\alpha_2}$ is a Lie triple system. It is easy to see that this Lie triple system is $J$-invariant. The complex totally geodesic submanifold of ${Q^n}^*$ generated by this Lie triple system is isometric to ${Q^{n-2}}^*$. However, the only complex totally geodesic submanifolds of a complex hyperbolic space are again complex hyperbolic spaces (see \cite{Wo63} and use duality). It follows that there exists a totally geodesic submanifold $\Sigma^{n-2} \cong \CC H^{n-2}(-4)$ of $P \cong \CC H^{n-1}(-4)$ with $T_o\Sigma^{n-2} = \cp_{\alpha_1+\alpha_2} \oplus \cp_{\alpha_2}$. We have
\[
T_oP^{n-1} = T_o\Sigma_1 \oplus T_o\Sigma^{n-2},\  \nu_oP^{n-1} = T_oB_1.
\]
The tangent space $T_o\Sigma^{n-2} = \cp_{\alpha_1+\alpha_2} \oplus \cp_{\alpha_2}$ and the normal space $\nu_o\Sigma^{n-2} = \ca \oplus \cp_{\alpha_1+2\alpha_2} \oplus \cp_{\alpha_1}$ are Lie triple systems in $\cp$. 

We summarize the previous discussion in the following theorem.

\begin{thm} \label{ghchP}
There exists a homogeneous complex hypersurface $P^{n-1}$ in $({Q^n}^*,g)$ which is isometric to the complex hyperbolic space $\CC H^{n-1}(-4)$ of constant holomorphic sectional curvature $-4$. In terms of root spaces and root vectors, the tangent space and normal space of $P^{n-1}$ at $o$ is
\[
T_oP^{n-1} = \RR H_{\alpha_1+2\alpha_2} \oplus \cp_{\alpha_1+2\alpha_2} \oplus \cp_{\alpha_1+\alpha_2} \oplus \cp_{\alpha_2},\ \nu_oP^{n-1} = \RR H_{\alpha_1} \oplus \cp_{\alpha_1}.
\]

The normal space $\nu_oP^{n-1}$ is a Lie triple system and the totally geodesic submanifold $B_1$ of ${Q^n}^*$ generated by this Lie triple system is isometric to a complex hyperbolic line $\CC H^1(-4)$ of constant (holomorphic) sectional curvature $-4$. The (non-zero) tangent vectors of $B_1$ are $\cA$-isotropic. In particular, the (non-zero) normal vectors of $P^{n-1}$ are $\cA$-isotropic singular tangent vectors of ${Q^n}^*$.

The tangent space $T_oP^{n-1}$ decomposes orthogonally into 
\[
T_oP^{n-1} = C(\nu_oP^{n-1}) \oplus (\cp_{\alpha_1+\alpha_2} \oplus \cp_{\alpha_2}),
\]
where $C$ is any real structure in $\cA_0$ at $o$.
The subspace $C(\nu_oP^{n-1})$ is a Lie triple system and the totally geodesic submanifold $\Sigma_1$ of ${Q^n}^*$ generated by this Lie triple system is isometric to a complex hyperbolic line $\CC H^1(-4)$ of constant (holomorphic) sectional curvature $-4$. The (non-zero) tangent vectors of $\Sigma_1$ are $\cA$-isotropic. The subspace $\cp_{\alpha_1+\alpha_2} \oplus \cp_{\alpha_2}$ is a Lie triple system in $\cp$ and a complex subspace of $T_oP^{n-1}$. The totally geodesic submanifold of ${Q^n}^*$ generated by this Lie triple system is isometric to the complex hyperbolic quadric ${Q^{n-2}}^*$ and the totally geodesic submanifold of $P^{n-1} \cong \CC H^{n-1}(-4)$ generated by this complex subspace is isometric to the complex hyperbolic space $\CC H^{n-2}(-4)$.

Let $\zeta \in \nu_oP^{n-1}$ be a unit normal vector of $P^{n-1}$. Then $\zeta$ is of the form
\[
\zeta = \frac{1}{2}\cos(\varphi)H_{\alpha_1} + \frac{1}{2}\sin(\varphi)JH_{\alpha_1}
\]
and the principal curvatures of $P^{n-1}$ with respect to $\zeta$ are $0,1,-1$ with corresponding principal curvature spaces
\begin{align*}
T^\zeta_0 & = C(\nu_oP^{n-1}) = T_o\Sigma_1,\\
T^\zeta_1 & = \{  \cos\left(\textstyle{\frac{\varphi}{2}}\right)X + \sin\left(\textstyle{\frac{\varphi}{2}}\right)JX : X \in \cp_{\alpha_1+\alpha_2}\} \subset V(C), \\
T^\zeta_{-1} & = \{\sin\left(\textstyle{\frac{\varphi}{2}}\right)X-  \cos\left(\textstyle{\frac{\varphi}{2}}\right)JX  : X \in \cp_{\alpha_1+\alpha_2}\} \subset JV(C),
\end{align*}
where $C = \cos(\varphi)C_0 + \sin(\varphi)JC_0$. The $0$-eigenspace is independent of the choice of unit normal vector $\zeta$  and coincides with the kernel $T_0$ of the shape operator of $P^{n-1}$. 
\end{thm}

Let $M$ be a submanifold of a Riemannian manifold $\bar{M}$ and $\zeta \in \nu_pM$ be a normal vector of $M$. Consider the Jacobi operator $\bar{R}_\zeta = \bar{R}(\cdot,\zeta)\zeta : T_p\bar{M} \to T_p\bar{M}$. If $\bar{R}_\zeta(T_pM) \subseteq T_pM$, then the restriction $\calK_\zeta$ of $\bar{R}_\zeta$ to $T_pM$ is a self-adjoint endomorphism of $T_pM$, the so-called normal Jacobi operator of $M$ with respect to $\zeta$. The family $\calK = (\calK_\zeta)_{\zeta \in \nu M}$ is called the normal Jacobi operator of $M$. 

A submanifold $M$ of a Riemannian manifold $\bar{M}$ is curvature-adapted if for every normal vector $\zeta \in \nu_pM$, $p \in M$, the following two conditions are satisfied:
\begin{enumerate}
\item[(i)] $\bar{R}_\zeta(T_pM) \subseteq T_pM$;
\item[(ii)] the normal Jacobi operator $\calK_\zeta$ and the shape operator $A_\zeta$ of $M$ are simultaneously diagonalizable, that is, 
\[
\calK_\zeta A_\zeta = A_\zeta \calK_\zeta.
\]
\end{enumerate}
Since $\bar{R}_{\lambda\zeta} = \lambda^2 \bar{R}_\zeta$ for all $\lambda > 0$, it suffices to check conditions (i) and (ii) only for unit normal vectors.
Curvature-adapted submanifolds were introduced in \cite{BV92}. They were also studied by Gray in \cite{G04} using the notion of compatible submanifolds. Curvature-adapted submanifolds form a very useful class of submanifolds in the context of focal sets and tubes. 

\begin{cor} \label{Pcurvad}
The homogeneous complex hypersurface $P^{n-1} \cong \CC H^{n-1}(-4)$ in $({Q^n}^*,g)$ is curvature-adapted.
\end{cor}

\begin{proof}
Let $\zeta$ be a unit normal vector of $P^{n-1}$ at $o$. Then $\zeta$ is an $\cA$-isotropic singular tangent vector of ${Q^n}^*$. We already computed the eigenvalues and eigenspaces of the Jacobi operator $\bar{R}_\zeta$ in Section \ref{tchq}. It follows from this that $\bar{R}_\zeta$ has three eigenvalues $0,-1,-4$ with corresponding eigenspaces
\[
E^\zeta_0  =  \RR \zeta \oplus \CC C_0\zeta\ ,\ 
E^\zeta_{-1}  = \cp \ominus (\CC \zeta \oplus \CC C_0\zeta)\ ,\ 
E^\zeta_{-4}  = \RR J\zeta.
\]
Note that $E^\zeta_{-1}$ is independent of the choice of the unit normal vector $\zeta$ and hence we can denote this space by $E_{-1}$.
The tangent space $T_oP^{n-1}$ is given by
\[
T_oP^{n-1} = C_0(\nu_oP^{n-1}) \oplus E_{-1}.
\]
From Theorem \ref{ghchP} we see that $T_0^\zeta = T_0 = C_0(\nu_oP^{n-1}) = \CC C_0\zeta \subset E^\zeta_0$ and $E_{-1} = T^\zeta_1 \oplus T^\zeta_{-1}$, which implies that $\calK_\zeta$ and $A_\zeta$ commute. Since this holds for all unit normal vectors $\zeta$, it follows that $P^{n-1}$ is curvature-adapted.
\end{proof}

\section{Tubes around the homogeneous complex hypersurface}
\label{thch}

In this section we discuss the geometry of the tubes around the homogeneous complex hypersurface $P^{n-1} \cong \CC H^{n-1}(-4)$ in $({Q^n}^*,g)$. We first observe that the tubes around $P^{n-1}$ are homogeneous real hypersurfaces in ${Q^n}^*$. In fact, the connected closed subgroup $H$ of $G = SO^0_{2,n}$ with Lie algebra
\[
\ch = \ck_{\alpha_1} \oplus \cd = \ck_{\alpha_1} \oplus \RR H_{\alpha_1+2\alpha_2} \oplus \cg_{\alpha_2} \oplus \cg_{\alpha_1+\alpha_2} \oplus \cg_{\alpha_1+2\alpha_2}
\]
acts on ${Q^n}^*$ with cohomogeneity one (see \cite{BD15}, Theorem 8). By construction, the orbit of $H$ containing $o$ is the homogeneous complex hypersurface $P^{n-1}$ and the principal orbits are the tubes around $P^{n-1}$. We denote by $P^{2n-1}_r$ the tube with radius $r \in \RR_+$ around $P^{n-1}$ in ${Q^n}^*$. Note that $P^{2n-1}_r$ is a homogeneous real hypersurface in ${Q^n}^*$ and hence $\dim_\RR(P^{2n-1}_r) = 2n-1$.

By Corollary \ref{Pcurvad}, the homogeneous complex hypersurface $P^{n-1}$ is curvature-adapted. Since tubes around curvature-adapted submanifolds in Riemannian symmetric spaces are again curvature-adapted (see \cite{G04}, Theorem 6.14, or \cite{BV92}, Theorem 6), Corollary \ref{Pcurvad} implies:

\begin{prop} 
The tube $P^{2n-1}_r$ with radius $r \in \RR_+$ around the homogeneous complex hypersurface $P^{n-1} \cong \CC H^{n-1}(-4)$ in $({Q^n}^*,g)$ is curvature-adapted.
\end{prop}

We can therefore use Jacobi field theory to compute the principal curvatures and principal curvature spaces of $P^{2n-1}_r$  (see e.g.\ \cite{BCO16}, Section 10.2.3, for a detailed description of the methodology). Since $P^{2n-1}_r$ is a homogeneous real hypersurface in ${Q^n}^*$, it suffices to compute the principal curvatures and principal curvature spaces at one point. Let $\zeta \in \nu_oP^{n-1}$ be a unit normal vector and $\gamma : \RR \to {Q^n}^*$ the geodesic in ${Q^n}^*$ with $\gamma(0) = o$ and $\dot\gamma(0) = \zeta$. Then $p = \gamma(r) \in P^{2n-1}_r$ and $\zeta_r = \dot\gamma(r)$ is a unit normal vector of $P^{2n-1}_r$ at $o$. Since $\zeta$ is $\cA$-isotropic, also $\zeta_r$ is $\cA$-isotropic. Thus the normal bundle of $P^{2n-1}_r$ consists of $\cA$-isotropic singular tangent vectors of ${Q^n}^*$.

We denote by $\gamma^\perp$ the parallel subbundle of the tangent bundle of ${Q^n}^*$ along $\gamma$ that is defined by the orthogonal complements of $\RR\dot{\gamma}(t)$ in $T_{\gamma(t)}{Q^n}^*$, $t \in \RR$, and put 
\[
\bar{R}^\perp_\gamma = \bar{R}_\gamma|_{\gamma^\perp} = \bar{R}(\cdot,\dot{\gamma})\dot{\gamma}|_{\gamma^\perp}.
\]
Let $D$ be the ${\rm End}(\gamma^\perp)$-valued tensor field along $\gamma$ solving the Jacobi equation
\[ 
D^{\prime\prime} + \bar{R}^\perp_\gamma \circ D = 0,\ 
D(0) = \begin{pmatrix} \id_{T_oP^{n-1}} & 0 \\ 0 & 0 \end{pmatrix},\ 
D^\prime(0) = \begin{pmatrix} -A_\zeta & 0 \\ 0 & \id_{\RR J\zeta} \end{pmatrix} ,
\]
where the decomposition of the matrices is with respect to the decomposition $\gamma^\perp(0) = T_oP^{n-1} \oplus \RR J\zeta$ and $A_\zeta$ is the shape operator of $P^{n-1}$ with respect to $\zeta$.
If $v \in T_oP^{n-1}$ and $B_v$ is the parallel vector field along $\gamma$ with $B_v(0) = v$, then $Z_v = DB_v$ is the Jacobi field along $\gamma$ with initial values $Z_v(0) = v$ and $Z_v^\prime(0) = -A_\zeta v$. If $v \in \RR J\zeta$ and $B_v$ is the parallel vector field along $\gamma$ with $B_v(0) = v$, then $Z_v = DB_v$ is the Jacobi field along $\gamma$ with initial values $Z_v(0) = 0$ and $Z_v^\prime(0) = v$. We decompose $T_oP^{n-1}$ orthogonally into $T_oP^{n-1} = T^\zeta_0 \oplus T^\zeta_1 \oplus T^\zeta_{-1}$ (see Theorem \ref{ghchP}). 

Since $\zeta$ is $\cA$-isotropic, the Jacobi operator $\bar{R}^\perp_\gamma$ at $o$ is of matrix form
\[
\begin{pmatrix}
0 & 0 & 0 & 0\\
0 & -1 & 0 & 0 \\
0 & 0 & -1 & 0 \\
0 & 0 & 0 & -4
\end{pmatrix}
\]
with respect to the decomposition $T^\zeta_0 \oplus T^\zeta_1 \oplus T^\zeta_{-1} \oplus \RR J\zeta$. Since $({Q^n}^*,g)$ is a Riemannian symmetric space, the Jacobi operator $\bar{R}^\perp_\gamma$ is parallel along $\gamma$. By solving the above second order initial value problem explicitly we obtain
\[
D(r) = \begin{pmatrix} 
1 & 0 & 0 & 0 \\
0 & e^{-r} & 0 & 0 \\
0 & 0 & e^r & 0 \\
0 & 0 & 0 & \frac{1}{2}\sinh(2r)
\end{pmatrix}
\]
with respect to the parallel translate of the decomposition $T^\zeta_0 \oplus T^\zeta_1 \oplus T^\zeta_{-1} \oplus \RR J\zeta$ along $\gamma$ from $o$ to $\gamma(r)$.
The shape operator $A^r_{\zeta_r}$ of $P^{2n-1}_r$ with respect to the unit normal vector $\zeta_r = \dot{\gamma}(r)$ satisfies the equation
\[ 
A^r_{\zeta_r} = - D^\prime(r) \circ D^{-1}(r).
\]
The matrix representation of $A^r_{\zeta_r}$ with respect to the parallel translate of the decomposition $T^\zeta_0 \oplus T^\zeta_1 \oplus T^\zeta_{-1} \oplus \RR J\zeta$ along $\gamma$ from $o$ to $\gamma(r)$ therefore is
\[
A^r_{\zeta_r} = \begin{pmatrix} 
0 & 0 & 0 & 0 \\
0 & 1 & 0 & 0 \\
0 & 0 & -1 & 0 \\
0 & 0 & 0 & -2\coth(2r)
\end{pmatrix}.
\]
It is remarkable that the principal curvatures of the tubes $P^{2n-1}_r$, corresponding to the maximal complex subbundle $\calC$, are the same as those for the focal set $P^{n-1}$. The only additional principal curvature comes from the circles in $P^{2n-1}_r$ generated by the unit normal bundle of $P^{n-1}$, which in fact is the Hopf principal curvature function $\alpha$. We change the orientation of the unit normal vector field of $P^{2n-1}_r$ so that $\alpha$ becomes positive, that is, $\alpha = 2\coth(2r)$. 

Since the K\"{a}hler structure $J$ is parallel along $\gamma$, the condition $JT^\zeta_1 = T^\zeta_{-1}$ is preserved by parallel translation along $\gamma$. From this we easily see that the shape operator $A^r_{\zeta_r}$ of $P^{2n-1}_r$ satisfies $A^r_{\zeta_r}\phi + \phi A^r_{\zeta_r} = 0$. We summarize the previous discussion in the following result.

\begin{thm}
Let $P^{2n-1}_r$ be the tube with radius $r \in \RR_+$ around the homogeneous complex hypersurface $P^{n-1} \cong \CC H^{n-1}(-4)$ in $({Q^n}^*,g)$. The normal bundle of $P^{2n-1}_r$ consists of $\cA$-isotropic singular tangent vectors of $({Q^n}^*,g)$.
The homogeneous real hypersurface $P^{2n-1}_r$ has four distinct constant principal curvatures
\[
0,1,-1,2\coth(2r)
\]
with multiplicities $2,n-2,n-2,1$, respectively, with respect to a suitable orientation of the unit normal vector field $\zeta_r$ of $P^{2n-1}_r$. In particular, the mean curvature of $P^{2n-1}_r$ is equal to $2\coth(2r)$. The corresponding principal curvature spaces are
\[
T^{\zeta_r}_0 = \CC C\zeta_r = \calC \ominus \calQ\ ,\ T^{\zeta_r}_{2\coth(2r)} = \RR J\zeta_r,
\]
where $C$ is an arbitrary real structure on ${Q^n}^*$. The principal curvature spaces $T^{\zeta_r}_1$ and $T^{\zeta_r}_{-1}$ are mapped into each other by the complex structure $J$ (or equivalently, by the structure tensor field $\phi$) and are contained in the $\pm 1$-eigenspaces of a suitable real structure $C$.
Moreover, the shape operator $A^r$ and the structure tensor field $\phi$ of $P^{2n-1}_r$ satisfy
\[
A^r \phi + \phi A^r = 0.
\]
\end{thm}

To put this into the context of Theorem \ref{mainthm}, we define $M_\alpha^{2n-1} = P_r^{2n-1}$ with $\alpha = 2\coth(2r)$.
Recall that the normal space $\nu_oP^{n-1}$ is a Lie triple system and the totally geodesic submanifold $B_1$ of ${Q^n}^*$ generated by this Lie triple system is isometric to a complex hyperbolic line $\CC H^1(-4)$ of constant holomorphic sectional curvature $-4$. The same is true for all the other normal spaces of $P^{n-1}$. It follows that, by construction, the integral curves of the Reeb vector field $\xi = -J\zeta$ are circles of radius $r$ in a complex hyperbolic line of constant sectional curvature $-4$. Such a circle has constant geodesic curvature $\alpha = 2\coth(2r)$. We thus see that the integral curves of the Reeb vector field are circles with radius $r$ in a complex hyperbolic line $\CC H^1(-4)$. This clarifies the geometric construction discussed in the introduction.

\section{The minimal homogeneous Hopf hypersurface}
\label{mhHs}

In this section we construct the minimal homogeneous real hypersurface $M^{2n-1}_0$ in $({Q^n}^*,g)$. The construction is a special case of the canonical extension technique developed by the author and Tamaru in \cite{BT13}. 

We start by defining the reductive subalgebra
\[
\cl_1  = \cg_{-\alpha_1} \oplus \cg_0 \oplus \cg_{\alpha_1} \cong \cs\cu_{1,1} \oplus \RR \oplus \cs\co_{n-2}
\]
and the nilpotent subalgebra
\[
\cn_1  = \cg_{\alpha_2} \oplus \cg_{\alpha_1 + \alpha_2} \oplus \cg_{\alpha_1 + 2\alpha_2} \cong \ch^{2n-3}
\]
of $\cg = \cs\co_{2,n}$. Here, $\ch^{2n-3}$ is the $(2n-3)$-dimensional Heisenberg algebra with $1$-dimensional center. Note that $\cn_1$ already appeared in the construction of the homogeneous complex hypersurface $P^{n-1}$ in Section \ref{hcs} as part of the subalgebra $\cd = \RR H_{\alpha_1+2\alpha_2} \oplus \cn_1$.
We define
\begin{equation*}
\ca_1 = \ker(\alpha_1) = \RR H_{\alpha_1+2\alpha_2} ,\ \ca^1 = \RR H_{\alpha_1},
\end{equation*}
which gives an orthogonal decomposition of $\ca$ into $\ca = \ca_1\oplus \ca^1$.
The reductive subalgebra $\cl_1$ is the centralizer and the normalizer of $\ca_1$ in $\cg$. Since $[\cl_1,\cn_1] \subseteq \cn_1$,
\begin{equation*}
\cq_1 = \cl_1 \oplus \cn_1
\end{equation*}
is a subalgebra of $\cg$, the so-called parabolic subalgebra of $\cg$ associated with the simple root $\alpha_1$. The subalgebra $\cl_1 = \cq_1 \cap \theta(\cq_1)$ is a reductive Levi subalgebra of $\cq_1$ and $\cn_1$ is the unipotent radical of $\cq_1$. Therefore the decomposition $\cq_1 = \cl_1 \oplus \cn_1$ is a semidirect sum of the Lie algebras $\cl_1$ and $\cn_1$. The decomposition $\cq_1 = \cl_1 \oplus \cn_1$ is the Chevalley decomposition of the parabolic subalgebra $\cq_1$.

Next, we define a reductive subalgebra $\cm_1$ of $\cg$ by
\begin{equation*}
\cm_1  =   \cg_{-\alpha_1} \oplus \ca^1 \oplus \cg_{\alpha_1} \oplus \ck_0 \cong \cs\cu_{1,1} \oplus \cs\co_{n-2}.
\end{equation*}
The subalgebra $\cm_1$ normalizes $\ca_1 \oplus \cn_1$. The decomposition
\begin{equation*}
\cq_1 = \cm_1 \oplus \ca_1 \oplus \cn_1
\end{equation*}
is  the Langlands decomposition of the parabolic subalgebra $\cq_1$. We define a subalgebra $\ck_1$ of $\ck$ by
\begin{equation*}
\ck_1 = \cq_1 \cap \ck = \cl_1 \cap \ck = \cm_1 \cap \ck = \ck_{\alpha_1} \oplus \ck_0  \cong \cs\co_2 \oplus \cs\co_{n-2} .
\end{equation*}

Next, we define the semisimple subalgebra
\[
\cg_1 = \cg_{-\alpha_1} \oplus \ca^1 \oplus \cg_{\alpha_1} \cong \cs\cu_{1,1}.
\]
It is easy to see that the subspaces
\begin{equation*}
\ca \oplus \cp_{\alpha_1} = \cl_1\cap \cp  ,\ 
\ca^1 \oplus \cp_{\alpha_1} = \cm_1 \cap \cp = \cg_1 \cap \cp 
\end{equation*}
are Lie triple systems in ${\mathfrak p}$. Then $\cg_1 = \ck_{\alpha_1} \oplus (\ca^1 \oplus \cp_{\alpha_1})$ is a Cartan decomposition of the semisimple subalgebra $\cg_1$ of $\cg$
and $\ca^1$ is a maximal abelian subspace of
$\ca^1 \oplus \cp_{\alpha_1}$. Moreover,  $\cg_1 = (\ck_{\alpha_1} \oplus \ca^1) \oplus
\cg_{-\alpha_1} \oplus \cg_{\alpha_1}$ is the restricted root space decomposition of
$\cg_1$ with respect to $\ca^1$ and $\{\pm\alpha_1\}$
is the corresponding set of restricted roots. 

We now relate these algebraic constructions to the geometry of the
complex hyperbolic quadric ${Q^n}^*$. We denote by $A_1 \cong \RR$ the connected abelian subgroup of $G$ with Lie algebra $\ca_1$ and by $N_1 \cong H^{2n-3}$ the
connected nilpotent subgroup of $G$ with Lie algebra $\cn_1 \cong \ch^{2n-3}$. Here, $H^{2n-3}$ is the $(2n-3)$-dimensional Heisenberg group with $1$-dimensional center. The centralizer $L_1 = Z_G(\ca_1) \cong SU_{1,1} \times \RR \times SO_{n-2}$ of
$\ca_1$ in $G$ is a reductive subgroup of $G$ with Lie
algebra $\cl_1$. The subgroup $A_1$ is contained in
the center of $L_1$. The subgroup $L_1$ normalizes $N_1$
and $Q_1 = L_1 N_1$ is a subgroup of $G$ with Lie algebra
$\cq_1$. The subgroup $Q_1$ coincides with the
normalizer $N_G(\cl_1 \oplus \cn_1)$ of
$\cl_1 \oplus \cn_1$ in $G$, and hence
$Q_1$ is a closed subgroup of $G$. The subgroup $Q_1$ is the
parabolic subgroup of $G$ associated with the simple root $\alpha_1$.

Let $G_1 \cong SU_{1,1}$ be the connected subgroup of $G$ with Lie algebra
$\cg_1 \cong \cs\cu_{1,1}$. The intersection $K_1$
of $L_1$ and $K$, i.e. $K_1 = L_1 \cap K \cong SO_2 \times SO_{n-2}$, is a maximal
compact subgroup of $L_1$ and $\ck_1$ is the Lie
algebra of $K_1$. The adjoint group $\Ad(L_1)$ normalizes
$\cg_1$, and consequently $M_1 = K_1 G_1 \cong SU_{1,1}  \times SO_{n-2}$ is a
subgroup of $L_1$. The Lie algebra of $M_1$ is
$\cm_1$ and $L_1$ is isomorphic to the Lie group
direct product $M_1 \times A_1$, i.e. $L_1 = M_1 \times
A_1 \cong (SU_{1,1}  \times SO_{n-2}) \times \RR$. The parabolic subgroup $Q_1$ acts transitively on ${Q^n}^*$ and the isotropy subgroup at $o$ is $K_1$, that is, ${Q^n}^* \cong
Q_1/K_1$.

Since $\cg_1 = \ck_{\alpha_1} \oplus (\ca^1 \oplus \cp_{\alpha_1})$ is a Cartan decomposition of the semisimple subalgebra $\cg_1$, we have $[\ca^1 \oplus \cp_{\alpha_1},\ca^1 \oplus \cp_{\alpha_1}] = \ck_{\alpha_1}$. Thus $G_1 \cong SU_{1,1}$ is the connected closed subgroup of $G$ with Lie algebra $[\ca^1 \oplus \cp_{\alpha_1},\ca^1 \oplus \cp_{\alpha_1}] \oplus
(\ca^1 \oplus \cp_{\alpha_1})$. Since $\ca^1 \oplus \cp_{\alpha_1}$ is a Lie triple
system in $\cp$, the orbit $B_1 = G_1 \cdot o$ of
the $G_1$-action on ${Q^n}^*$ containing $o$ is a connected totally
geodesic submanifold of ${Q^n}^*$ with $T_oB_1 = \ca^1 \oplus \cp_{\alpha_1}$. Moreover, $B_1$ is a Riemannian symmetric space of non-compact
type and rank $1$, and
\begin{equation*}
B_1 = G_1 \cdot o = G_1/(G_1\cap K_1) \cong SU_{1,1}/SO_2 \cong \CC H^1(-4),
\end{equation*}
where $\CC H^1(-4)$ is a complex hyperbolic line of constant (holomorphic) sectional curvaturer $-4$. The submanifold $B_1$ is a boundary component of ${Q^n}^*$ in the context of the maximal Satake compactification of  ${Q^n}^*$. This boundary component coincides with the totally geodesic submanifold $B_1$ that we constructed in Section \ref{hcs}.

Clearly, $\ca_1$ is a Lie triple system and
the corresponding totally geodesic submanifold is a Euclidean line 
$\RR = A_1 \cdot o$.
Since the action of $A_1$ on $M$ is free and $A_1$ is simply
connected, we can identify $\RR$ and $A_1$ canonically. 

Finally, $\cf_1 = \ca \oplus \cp_{\alpha_1}$ is a Lie triple system and the corresponding totally
geodesic submanifold $F_1$ is the symmetric space
\begin{equation*}
F_1 = L_1 \cdot o = L_1/K_1= (M_1 \times
A_1)/K_1 = B_1 \times \RR \cong \CC H^1(-4) \times \RR.
\end{equation*}

The submanifolds $F_1$ and $B_1$ have a natural geometric
interpretation. Denote by $\bar{C}^+(\Lambda) \subset \ca$
the closed positive Weyl chamber that is determined by the two simple
roots $\alpha_1$ and $\alpha_2$. Let $Z$ be non-zero vector in $\bar{C}^+(\Lambda)$
such that $\alpha_1(Z) = 0$ and $\alpha_2(Z) >
0$, and consider the
geodesic $\gamma_Z(t) = \Exp(tZ) \cdot o$ in ${Q^n}^*$ with
$\gamma_Z(0) = o$ and $\dot{\gamma}_Z(0) = Z$. The totally geodesic
submanifold $F_1$ is the union of all geodesics in ${Q^n}^*$ 
parallel to $\gamma_Z$, and $B_1$ is the semisimple part of
$F_1$ in the de Rham decomposition of $F_1$ (see e.g.\
\cite{Eb96}, Proposition 2.11.4  and Proposition 2.20.10).

The parabolic group $Q_1$ is diffeomorphic to the product $M_1 \times
A_1 \times N_1$. This analytic diffeomorphism induces an
analytic diffeomorphism between 
\[
B_1 \times \RR \times N_1 \cong \CC H^1(-4) \times \RR \times H^{2n-3}
\] 
and ${Q^n}^*$, giving a horospherical decomposition of the complex hyperbolic quadric ${Q^n}^*$,
\[
\CC H^1(-4) \times \RR \times H^{2n-3} \cong {Q^n}^*.
\]
The factor $\RR \times H^{2n-3}$ corresponds to the homogeneous complex hypersurface $P^{n-1} \cong \CC H^{n-1}(-4)$ that we discussed in Section \ref{hcs}.

We have $\RR H_{\alpha_1} = \ca^1 \subset \cg_1$ and $G_1 \cdot o = B_1$. It follows from Theorem \ref{ghchP} that $\ca^1$ consists of $\cA$-isotropic tangent vectors of ${Q^n}^*$. Let $A^1 \cong \RR$ be the abelian subalgebra of $\ca$ with Lie algebra $\ca^1$. Then the orbit $A^1 \cdot o$ is the path of an $\cA$-isotropic geodesic $\gamma$ (determined by the root vector $H_{\alpha_1}$) in the complex hyperbolic quadric ${Q^n}^*$. Moreover, by construction, this geodesic is contained in the boundary component $B_1 \cong \CC H^1(-4)$. The action of $A^1$ on $\CC H^1(-4)$ is of cohomogeneity one. The orbit containing $o$ is the geodesic $\gamma$, and the other orbits are the equidistant curves to $\gamma$.

The canonical extension of the cohomogeneity one action of $A^1$ on the boundary component $B_1 \cong \CC H^1(-4)$ is defined as follows. We first define the solvable subalgebra
\begin{align*}
\cs_1 & = \ca^1 \oplus \ca_1 \oplus \cn_1 = \ca \oplus \cn_1 \\
& = \left\{  
\begin{pmatrix}
0 & y & a_1 & -y & v_1 & \cdots &  v_{n-2}\\
-y & 0 & y & a_2 & w_1 & \cdots &  w_{n-2}  \\
a_1 & y & 0 & -y & v_1 & \cdots &  v_{n-2}\\
-y & a_2 & y & 0 & w_1 & \cdots &  w_{n-2} \\
v_1 & w_1 & -v_1 & -w_1 & 0 & \cdots & 0 \\
\vdots & \vdots & \vdots & \vdots & \vdots & \ddots & \vdots  \\
v_{n-2} & w_{n-2} & -v_{n-2} & -w_{n-2} & 0 & \cdots & 0 
\end{pmatrix} : 
\begin{array}{l} a_1,a_2,y \in \RR,\\ v,w \in \RR^{n-2} \end{array} \right\}
\end{align*}
of $\ca \oplus \cn$. Let $S_1$ be the connected solvable subgroup of $AN$ with Lie algebra $\cs_1$. Then the action of $S_1$ on $AN$ (resp.\ ${Q^n}^*$) is of cohomogeneity one (see \cite{BT13}). By construction, all orbits of the $S_1$-action on $AN$ (resp.\ ${Q^n}^*$) are homogeneous real hypersurfaces in $(AN,\langle \cdot , \cdot \rangle)$ (resp.\ $({Q^n}^*,g)$). Let $\hat{M}^{2n-1}_0$ (resp.\ $M^{2n-1}_0$) be the orbit containing the point $o$. Geometrically, we can describe this orbit as the canonical extension of an $\cA$-isotropic geodesic in the boundary component $B_1 \cong \CC H^1(-4)$.

We will now compute the shape operator of the homogeneous real hypersurface $\hat{M}^{2n-1}_0$ in $(AN,\langle \cdot , \cdot \rangle)$. Since $\hat{M}^{2n-1}_0$ is homogeneous, it suffices to make the computations at the point $o$.
We define
\[
\hat\zeta = 
\begin{pmatrix} 
0 & 1 & 0 & 1 & 0 & \cdots & 0 \\
-1 & 0 & 1 & 0 & 0 & \cdots & 0 \\
0 & 1 & 0 & 1 & 0 & \cdots & 0 \\
1 & 0 & -1 & 0 &  0 & \cdots & 0 \\
0 & 0 & 0 & 0 & 0 & \cdots & 0 \\
\vdots & \vdots & \vdots & \vdots & \vdots & \ddots & \vdots\\
0 & 0 & 0 & 0 & 0 & \cdots & 0
\end{pmatrix}
\in \cg_{\alpha_1} \subset \cn.
\]
Then 
\[
\theta(\hat\zeta) =  
\begin{pmatrix} 
0 & 1 & 0 & -1 & 0 & \cdots & 0 \\
-1 & 0 & -1 & 0 & 0 & \cdots & 0 \\
0 & -1 & 0 & 1 & 0 & \cdots & 0 \\
-1 & 0 & -1 & 0 &  0 & \cdots & 0 \\
0 & 0 & 0 & 0 & 0 & \cdots & 0 \\
\vdots & \vdots & \vdots & \vdots & \vdots & \ddots & \vdots\\
0 & 0 & 0 & 0 & 0 & \cdots & 0
\end{pmatrix}
\in \cg_{-\alpha_1}
\]
and 
\begin{align*}
\langle \hat\zeta , \hat\zeta \rangle & = - \frac{1}{8}\tr(\hat\zeta\theta(\hat\zeta)) = 1.
\end{align*}
Thus $\hat\zeta$ is a unit normal vector of $\hat{M}^{2n-1}_0$ at $o$.
Let $\hat{A}$ be the shape operator of $\hat{M}^{2n-1}_0$ in $(AN,\langle \cdot , \cdot \rangle)$ with respect to $\hat{\zeta}$. As in Section \ref{hcs}, using arguments involving the Weingarten and Koszul formulas, we can show that 
\[
\hat{A} \hat{X} = [\zeta,\hat{X}]_{\cs_1}
\]
for all $\hat{X} \in \cs_1$, where $\zeta = \frac{1}{2}(\hat\zeta - \theta(\hat\zeta))$ is the orthogonal projection of $\hat{\zeta}$ onto $\cp$ and $[\, \cdot\ ]_{\cs_1}$ is the orthogonal projection onto $\cs_1$.

We have
\[
\zeta = \frac{1}{2}(\hat\zeta - \theta(\hat\zeta)) = 
\begin{pmatrix} 
0 & 0 & 0 & 1 & 0 & \cdots & 0 \\
0 & 0 & 1 & 0 & 0 & \cdots & 0 \\
0 & 1 & 0 & 0 & 0 & \cdots & 0 \\
1 & 0 & 0 & 0 &  0 & \cdots & 0 \\
0 & 0 & 0 & 0 & 0 & \cdots & 0 \\
\vdots & \vdots & \vdots & \vdots & \vdots & \ddots & \vdots\\
0 & 0 & 0 & 0 & 0 & \cdots & 0
\end{pmatrix}
\in \cp_{\alpha_1}.
\]
For
\[
\hat{X} = \begin{pmatrix}
0 & y & a_1 & -y & v_1 & \cdots &  v_{n-2}\\
-y & 0 & y & a_2 & w_1 & \cdots &  w_{n-2}  \\
a_1 & y & 0 & -y & v_1 & \cdots &  v_{n-2}\\
-y & a_2 & y & 0 & w_1 & \cdots &  w_{n-2} \\
v_1 & w_1 & -v_1 & -w_1 & 0 & \cdots & 0  \\
\vdots & \vdots & \vdots & \vdots & \vdots & \ddots & \vdots  \\
v_{n-2} & w_{n-2} & -v_{n-2} & -w_{n-2} & 0 & \cdots & 0 
\end{pmatrix} \in \cs_1
\]
we then compute
\[
[\zeta,\hat{X}] = 
 \begin{pmatrix}
0 & a_2-a_1 & 0 & 0 & w_1 & \cdots &  w_{n-2} \\
a_1-a_2 & 0 & 0 & 0 & v_1 & \cdots &  v_{n-2} \\
0 & 0 & 0 & a_2-a_1 & w_1 & \cdots &  w_{n-2}\\
0 & 0 & a_1-a_2 & 0 & v_1 & \cdots &  v_{n-2} \\
w_1 & v_1 & -w_1 & -v_1 & 0 & \cdots & 0 \\
\vdots & \vdots & \vdots & \vdots & \vdots & \ddots & \vdots  \\
w_{n-2} & v_{n-2} & -w_{n-2} & -v_{n-2} & 0 & \cdots & 0   
\end{pmatrix}.
\]
The orthogonal projection of $[\zeta,\hat{X}]$ onto $\cs_1$ is
\[
[\zeta,\hat{X}]_{\cs_1} = 
 \begin{pmatrix}
0 & 0 & 0 & 0 & w_1 & \cdots &  w_{n-2} \\
0 & 0 & 0 & 0 & v_1 & \cdots &  v_{n-2} \\
0 & 0 & 0 & 0 & w_1 & \cdots &  w_{n-2}\\
0 & 0 & 0 & 0 & v_1 & \cdots &  v_{n-2} \\
w_1 & v_1 & -w_1 & -v_1 & 0 & \cdots &  0\\
\vdots & \vdots & \vdots & \vdots & \vdots & \ddots &  \vdots \\
w_{n-2} & v_{n-2} & -w_{n-2} & -v_{n-2} & 0 & \cdots & 0   
\end{pmatrix}.
\]	
We conclude that the shape operator $\hat{A}$ of $\hat{M}^{2n-1}_0$ in $(AN,\langle \cdot , \cdot \rangle)$ with respect to $\hat{\zeta}$ is given by
\[
\hat{A} \hat{X} = 
 \begin{pmatrix}
0 & 0 & 0 & 0 & w_1 & \cdots &  w_{n-2} \\
0 & 0 & 0 & 0 & v_1 & \cdots &  v_{n-2} \\
0 & 0 & 0 & 0 & w_1 & \cdots &  w_{n-2}\\
0 & 0 & 0 & 0 & v_1 & \cdots &  v_{n-2} \\
w_1 & v_1 & -w_1 & -v_1 & 0 & \cdots  & 0\\
\vdots & \vdots & \vdots & \vdots & \vdots & \ddots & \vdots  \\
w_{n-2} & v_{n-2} & -w_{n-2} & -v_{n-2} & 0 & \cdots & 0 
\end{pmatrix}
\]
with
\[
\hat{X} = \begin{pmatrix}
0 & y & a_1 & -y & v_1 & \cdots &  v_{n-2}\\
-y & 0 & y & a_2 & w_1 & \cdots &  w_{n-2}  \\
a_1 & y & 0 & -y & v_1 & \cdots &  v_{n-2}\\
-y & a_2 & y & 0 & w_1 & \cdots &  w_{n-2} \\
v_1 & w_1 & -v_1 & -w_1 & 0 & \cdots & 0  \\
\vdots & \vdots & \vdots & \vdots & \vdots & \ddots & \vdots  \\
v_{n-2} & w_{n-2} & -v_{n-2} & -w_{n-2} & 0 & \cdots & 0
\end{pmatrix} \in \cs_1.
\]
From this we deduce that $0$ is a principal curvature of $\hat{M}^{2n-1}_0$ with multiplicity $3$ and  corresponding principal curvature space
\begin{align*}
\hat{T}_0 &  = \ca \oplus \cg_{\alpha_1\oplus 2\alpha_2}.
\end{align*}
On the orthogonal complement $\cg_{\alpha_1+\alpha_2} \oplus \cg_{\alpha_2}$ the shape operator is of the form
\[
\hat{A} = \begin{pmatrix} 
0 & 1 \\ 1 & 0
\end{pmatrix}
\]
with respect to the orthogonal decomposition $\cg_{\alpha_1+\alpha_2} \oplus \cg_{\alpha_2}$. The characteristic polynomial of this matrix is $x^2 - 1$, and hence the eigenvalues of $\hat{A}$ restricted to $\cg_{\alpha_1+\alpha_2} \oplus \cg_{\alpha_2}$ are $1$ and $-1$. The corresponding eigenspaces are
\[
\hat{T}_1 = \left\{ \begin{pmatrix}
0 & 0 & 0 & 0 & u_1 & \cdots &  u_{n-2}\\
0 & 0 & 0 & 0 & u_1 & \cdots &  u_{n-2}  \\
0 & 0 & 0 & 0 & u_1 & \cdots &  u_{n-2}\\
0 & 0 & 0 & 0 & u_1 & \cdots &  u_{n-2} \\
u_1 & u_1 & -u_1 & -u_1 & 0 & \cdots &  0 \\
\vdots & \vdots & \vdots & \vdots & \vdots & \ddots & \vdots \\
u_{n-2} & u_{n-2} & -u_{n-2} & -u_{n-2} & 0 & \cdots & 0
\end{pmatrix} \right\} \cong \RR^{n-2}
\]
and
\[
\hat{T}_{-1} = \left\{ \begin{pmatrix}
0 & 0 & 0 & 0 & u_1 & \cdots &  u_{n-2}\\
0 & 0 & 0 & 0 & -u_1 & \cdots &  -u_{n-2}  \\
0 & 0 & 0 & 0 & u_1 & \cdots &  u_{n-2}\\
0 & 0 & 0 & 0 & -u_1 & \cdots &  -u_{n-2} \\
u_1 & -u_1 & -u_1 & u_1 & 0 & \cdots & 0 \\
\vdots & \vdots & \vdots & \vdots & \vdots & \ddots & \vdots  \\
u_{n-2} & -u_{n-2} & -u_{n-2} & u_{n-2} & 0 & \cdots & 0
\end{pmatrix} \right\} \cong \RR^{n-2}.
\]

All of the above calculations are with respect to the metric $\langle \cdot , \cdot \rangle$ on $AN$. We now switch to the Riemannian metric $g$ on ${Q^n}^*$ and the Cartan decomposition $\cg = \ck \oplus \cp$. Recall that, by construction, $(AN,\langle \cdot , \cdot \rangle)$ and $({Q^n}^*,g)$ are isometric and the metrics are related by
\[
\langle H_1 + \hat{X}_1 , H_2 + \hat{X}_2 \rangle  
= g(H_1,H_2) + g(X_1,X_2)
\]
with $H_1,H_2 \in \ca$ and $\hat{X}_1,\hat{X}_2 \in \cn$.

Since $\hat{\zeta}$ is a unit vector in $\cg_{\alpha_1}$, the vector $\zeta = \frac{1}{2}(\hat\zeta - \theta(\hat\zeta))$ is a unit vector in $\cp_{\alpha_1}$. Since $\cp_{\alpha_1} \subset T_oB_1$ and all (non-zero) tangent vectors of the boundary component $B_1$ are $\cA$-isotropic (see Theorem \ref{ghchP}), we conclude that the normal bundle of $M^{2n-1}_0$ consists of $\cA$-isotropic singular tangent vectors of $({Q^n}^*,g)$.

Let $A$ be the shape operator of $M^{2n-1}_0$ in $({Q^n}^*,g)$ with respect to $\zeta$. The above calculations  imply that
\[
A X = 
 \begin{pmatrix}
0 & 0 & 0 & 0 & w_1 & \cdots & w_{n-2} \\
0 & 0 & 0 & 0 & v_1 & \cdots  & v_{n-2} \\
0 & 0 & 0 & 0 & 0 & \cdots &  0 \\
0 & 0 & 0 & 0 & 0 & \cdots &  0 \\
w_1 & v_1 & 0 & 0 & 0 & \cdots & 0 \\
\vdots & \vdots & \vdots & \vdots & \vdots & \ddots &  \vdots \\
w_{n-2} & v_{n-2} & 0 & 0 & 0 & \cdots & 0 
\end{pmatrix}
\]
with
\[
X = \begin{pmatrix}
0 & 0 & a_1 & -y & v_1 & \cdots  & v_{n-2}\\
0 & 0 & y & a_2 & w_1 & \cdots &  w_{n-2}  \\
a_1 & y & 0 & 0 &  0 & \cdots &  0 \\
-y & a_2 & 0 & 0 & 0 & \cdots & 0  \\
v_1 & w_1 & 0 & 0 & 0 & \cdots & 0  \\
\vdots & \vdots & \vdots & \vdots & \vdots & \ddots & \vdots  \\
v_{n-2} & w_{n-2} & 0 & 0 & 0 & \cdots & 0 
\end{pmatrix} \in T_oM^{2n-1}_0 \subset \cp.
\]
From this we easily deduce the following result.

\begin{thm} \label{gEso}
Let $M^{2n-1}_0$ be the homogeneous real hypersurface in $({Q^n}^*,g)$ obtained by canonical extension of the geodesic that is tangent to the root vector $H_{\alpha_1}$ in the boundary component $B_1 \cong \CC H^1(-4)$ of $({Q^n}^*,g)$. The normal bundle of $M^{2n-1}_0$ consists of $\cA$-isotropic singular tangent vectors of $({Q^n}^*,g)$ and $M^{2n-1}_0$ has three distinct constant principal curvatures $0$, $1$, $-1$ with multiplicities $3$, $n-2$, $n-2$, respectively. The principal curvature spaces $T_0$, $T_1$ and $T_{-1}$ are
\begin{align*}
T_0 &= \ca \oplus \cp_{\alpha_1+2\alpha_2} = \RR J\zeta \oplus (\calC \ominus \calQ),\\
T_1 & = \{ X-JX : X \in \cp_{\alpha_2}\} = \{ X+JX : X \in \cp_{\alpha_1+\alpha_2}\},\\
T_{-1} & = \{ X+JX : X \in \cp_{\alpha_2}\} = \{ X-JX : X \in \cp_{\alpha_1+\alpha_2}\}.
\end{align*}
We have $T_1 \oplus T_{-1} = \calQ$ and $JT_1 = T_{-1}$.
The shape operator $A$ of $M^{2n-1}_0$ satisfies
\[
A\phi + \phi A = 0.
\]
\end{thm}

Note that
\[
T_1 \subset V\left( \frac{1}{\sqrt{2}}(C_0 + JC_0) \right),\  T_{-1} \subset JV\left( \frac{1}{\sqrt{2}}(C_0 + JC_0) \right).
\]

We immediately see from Theorem \ref{gEso} that $\tr(A) = 0$.

\begin{cor} 
The homogeneous Hopf hypersurface $M^{2n-1}_0$ in $({Q^n}^*,g)$ is minimal.
\end{cor}

The eigenspaces $T_0$, $T_1$ and $T_{-1}$ of the shape operator $A$ and the eigenspaces $E_0$, $E_{-1}$ and $E_{-4}$ of the normal Jacobi operator $\calK_\zeta$ satisfy
\[
T_0 = E_0 \oplus E_{-4}\ ,\ T_{-1} \oplus T_1 = E_{-1}.
\]
It follows that $A$ and $\calK = \calK_\zeta$ are simultaneously diagonalizable and hence $A\calK = \calK A$. This implies that $M^{2n-1}_0$ is curvature-adapted.

\begin{cor} \label{gEca}
The homogeneous Hopf hypersurface $M^{2n-1}_0$ in $({Q^n}^*,g)$ is curvature-adapted.
\end{cor}

We finally relate this construction to the discussion in the introduction. The subalgebra 
\[
\cs_1 = \ca \oplus \cn_1 = \ca \oplus  \cg_{\alpha_2} \oplus \cg_{\alpha_1 + \alpha_2} \oplus \cg_{\alpha_1 + 2\alpha_2}
\]
of $\ca \oplus \cn$ contains the subalgebra
\[
\cd =  \RR H_{\alpha_1+2\alpha_2} \oplus \cg_{\alpha_2} \oplus \cg_{\alpha_1+\alpha_2} \oplus \cg_{\alpha_1+2\alpha_2}.
\] 
The subalgebra $\cd$ induces the homogeneous complex hypersurface $\hat{P}^{n-1} \cong \CC H^{n-1}(-4)$, as discussed in Section \ref{hcs}. Since the construction is left-invariant, it follows that the homogeneous real hypersurface $\hat{M}^{2n-1}_0$ in $(AN,\langle \cdot , \cdot \rangle)$ is foliated by isometric copies of the homogeneous complex hypersurface $\hat{P}^{n-1} \cong \CC H^{n-1}(-4)$. This implies that the homogeneous complex hypersurface $M^{2n-1}_0$ in $({Q^n}^*,g)$ is foliated by isometric copies of the homogeneous complex hypersurface $P^{n-1} \cong \CC H^{n-1}(-4)$. The normal space $\nu_oP^{n-1}$ is a Lie triple system and the totally geodesic submanifold $B_1$ of ${Q^n}^*$ generated by this Lie triple system is a complex hyperbolic line $\CC H^1(-4)$ of constant holomorphic sectional curvature $-4$. The same is true for all the normal spaces of $P^{n-1}$ at other points. It follows that, by construction, the integral curves of the Reeb vector field $\xi = -J\zeta$ are geodesics in a complex hyperbolic line of constant (holomorphic) sectional curvature $-4$. Such a geodesic has constant geodesic curvature $0$. This clarifies the geometric construction explained in the introduction.

\section{Equidistant real hypersurfaces}
\label{eqdihyp}

In this section we compute the shape operator of the other orbits of the cohomogeneity one action on $({Q^n}^*,g)$ that we discussed in Section \ref{mhHs}. Recall that $M^{2n-1}_0$ is the orbit of this action containing $o$. Since the action is isometric, the other orbits are the equidistant real hypersurfaces to $M^{2n-1}_0$. For $r \in \RR_+$ we denote by $M^{2n-1}_\alpha$ the equidistant real hypersurface to $M^{2n-1}_0$ at oriented distance $r \in \RR_+$, where we put $\alpha = 2\tanh(2r)$.

From Corollary \ref{gEca} we know that $M^{2n-1}_0$ is a curvature-adapted real hypersurface in ${Q^n}^*$.  We can therefore use Jacobi field theory to compute the principal curvatures and principal curvature spaces of  $M^{2n-1}_\alpha$ (see e.g.\ \cite{BCO16}, Section 10.2.2). Since $M^{2n-1}_\alpha$ is a homogeneous real hypersurface in ${Q^n}^*$, it suffices to compute the principal curvatures and principal curvature spaces at one point. Let $\zeta \in \nu_oM^{2n-1}_0$ be the unit normal vector of $M^{2n-1}_0$ as defined in Section \ref{mhHs} and $A_\zeta$ be the shape operator of $M^{2n-1}_0$ at $o$ with respect to $\zeta$. We denote by $T_0$, $T_1$ and $T_{-1}$ the principal curvature spaces as in Theorem \ref{gEso}. Let $\gamma : \RR \to {Q^n}^*$ be the geodesic in ${Q^n}^*$ with $\gamma(0) = o$ and $\dot\gamma(0) = \zeta$. Then $p = \gamma(r) \in M^{2n-1}_\alpha$ and $\zeta_r = \dot\gamma(r)$ is a unit normal vector of $M^{2n-1}_\alpha$ at $p$. We denote by $\gamma^\perp$ the parallel subbundle of the tangent bundle of ${Q^n}^*$ along $\gamma$ that is defined by the orthogonal complements of $\RR\dot{\gamma}(t)$ in $T_{\gamma(t)}{Q^n}^*$, and put 
\[
\bar{R}^\perp_\gamma = \bar{R}_\gamma|_{\gamma^\perp} = \bar{R}(\cdot,\dot{\gamma})\dot{\gamma}|_{\gamma^\perp}.
\]
Let $D$ be the ${\rm End}(\gamma^\perp)$-valued tensor field along $\gamma$ solving the Jacobi equation
\[ 
D^{\prime\prime} + \bar{R}^\perp_\gamma \circ D = 0 ,\ 
D(0) =  \id_{T_oM^{2n-1}_0} ,\ 
D^\prime(0) =  -A_\zeta.
\]
If $v \in T_oM^{2n-1}_0$ and $B_v$ is the parallel vector field along $\gamma$ with $B_v(0) = v$, then $Y = DB_v$ is the Jacobi field along $\gamma$ with initial values $Y(0) = v$ and $Y^\prime(0) = -A_\zeta v$.

Since $\zeta$ is $\cA$-isotropic, the Jacobi operator $\bar{R}^\perp_\gamma$ at $o$ is of matrix form
\[
\begin{pmatrix}
0 & 0 & 0 & 0\\
0 & -1 & 0 & 0 \\
0 & 0 & -1 & 0 \\
0 & 0 & 0 & -4
\end{pmatrix}
\]
with respect to the decomposition $\CC C\zeta \oplus T_1 \oplus T_{-1} \oplus \RR J\zeta$. Note that $T_0 = \CC C\zeta \oplus  \RR J\zeta$. Since $({Q^n}^*,g)$ is a Riemannian symmetric space, the Jacobi operator $\bar{R}^\perp_\gamma$ is parallel along $\gamma$. By solving the above second order initial value problem explicitly we obtain
\[
D(r) = \begin{pmatrix} 
1 & 0 & 0 & 0 \\
0 & e^{-r} & 0 & 0 \\
0 & 0 & e^r & 0 \\
0 & 0 & 0 & \cosh(2r)
\end{pmatrix}
\]
with respect to the parallel translate of the decomposition $T_0 \oplus T_1 \oplus T_{-1} \oplus \RR J\zeta$ along $\gamma$ from $o$ to $\gamma(r)$.
The shape operator $A^\alpha_{\zeta_r}$ of $M^{2n-1}_\alpha$ with respect to $\zeta_r = \dot{\gamma}(r)$ satisfies the equation
\[ 
A^\alpha_{\zeta_r} = - D^\prime(r) \circ D^{-1}(r).
\]
The matrix representation of $A^\alpha_{\zeta_r}$ with respect to the parallel translate of the decomposition $T_0 \oplus T_1 \oplus T_{-1} \oplus \RR J\zeta$ along $\gamma$ from $o$ to $\gamma(r)$ therefore is
\[
A^\alpha_{\zeta_r} = \begin{pmatrix} 
0 & 0 & 0 & 0 \\
0 & 1 & 0 & 0 \\
0 & 0 & -1 & 0 \\
0 & 0 & 0 & -2\tanh(2r)
\end{pmatrix}.
\]
It is remarkable that the principal curvatures of the equidistant real hypersurfaces to $M^{2n-1}_0$ are preserved along the parallel translate of the maximal complex subspace $\calC_o \subset T_oM_0^{2n-1}$. The only additional principal curvature arises in direction of the Reeb vector field, which is the Hopf principal curvature. We change the orientation of the unit normal vector field $\zeta_r$ so that the Hopf principal curvature is positive, that is, is equal to $\alpha$. Thus we have proved:

\begin{thm} 
Let $M^{2n-1}_0$ be the minimal homogeneous Hopf hypersurface in $({Q^n}^*,g)$ as in Section \ref{mhHs} and $M^{2n-1}_\alpha$ be the equidistant real hypersurface at oriented distance $r \in \RR_+$ from $M^{2n-1}_0$, where $\alpha = 2\tanh(2r)$. Then $M^{2n-1}_\alpha$ is a homogeneous Hopf hypersurface with four distinct constant principal curvatures $0$, $1$, $-1$, $2\tanh(2r)$ with multiplicities $2$, $n-2$, $n-2$, $1$, respectively. 
The principal curvature spaces $T_0$, $T_1$, $T_{-1}$, $T_{2\tanh(2r)}$ satisfy
\begin{align*}
T_0 &= \calC \ominus \calQ,\\
T_{2\tanh(2r)} & = \RR J\zeta_r = \calC^\perp,\\
T_1 \oplus T_{-1} & = \calQ \mbox{ and } JT_1 = T_{-1}.
\end{align*}
Moreover, the shape operator $A^\alpha$ and the structure tensor field $\phi$ of $M^{2n-1}_\alpha$ satisfy
\[
A^\alpha \phi + \phi A^\alpha = 0.
\]
\end{thm}

The principal curvature spaces $T_{2\tanh(2r)}$, $T_0$, $T_1$ and $T_{-1}$ of the shape operator $A^\alpha$ and the eigenspaces $E_0$, $E_{-1}$ and $E_{-4}$ of the normal Jacobi operator $\calK^\alpha = \calK_{\zeta_r}$ satisfy
\[
T_{2\tanh(2r)} = E_{-4}\ ,\ T_0 = E_0 \ ,\ T_{-1} \oplus T_1 = E_{-1}.
\]
It follows that $A^\alpha$ and $\calK^\alpha$ are simultaneously diagonalizable and hence $A^\alpha \calK^\alpha = \calK^\alpha A^\alpha$. This implies:

\begin{cor} 
The equidistant real hypersurface $M^{2n-1}_\alpha$, $0 < \alpha < 2$, to the minimal homogeneous Hopf hypersurface $M^{2n-1}_0$ in $({Q^n}^*,g)$ is curvature-adapted.
\end{cor}

By construction, the integral curves of the Reeb vector field on $M^{2n-1}_\alpha$ are congruent to an equidistant curve at distance $r = \frac{1}{2}\tanh^{-1}(\frac{\alpha}{2})$ to a geodesic in a complex hyperbolic line $\CC H^1(-4)$. Such an equidistant curve has constant geodesic curvature $2\tanh(2r)$. As in previous cases, this leads to the geometric interpretation of $M^{2n-1}_\alpha$ given by attaching copies of the homogeneous complex hypersurface $P^{n-1} \cong \CC H^{n-1}(-4)$ to such an equidistant curve to a geodesic in the boundary component $B_1 \cong \CC H^1(-4)$. Equivalently, $M^{2n-1}_\alpha$ is the canonical extension of an equidistant curve at distance $r = \frac{1}{2}\tanh^{-1}(\frac{\alpha}{2})$ to a geodesic in the boundary component $B_1 \cong \CC H^1(-4)$.

\section{The homogeneous Hopf hypersurface of horocyclic type}
\label{canexthoro}

In this section we discuss the canonical extension of a horocycle in the boundary component $B_1 \cong \CC H^1(-4)$, which leads to the homogeneous real hypersurface $M^{2n-1}_2$ in Theorem \ref{mainthm}. We first define the solvable subalgebra 
\[
\ch_1 = \ca_1 \oplus \cn
 = 
\left\{ 
\begin{pmatrix}
0 & x+y & a & x-y & v_1 & \cdots & v_{n-2} \\
-x-y & 0 & x+y & a & w_1 & \cdots & w_{n-2} \\
a & x+y & 0 & x-y & v_1 & \cdots & v_{n-2} \\
x-y & a & -x+y & 0 & w_1 & \cdots & w_{n-2} \\
v_1 & w_1 & -v_1 & -w_1 & 0 & \cdots & 0 \\
\vdots & \vdots & \vdots & \vdots & \vdots & \ddots & \vdots\\
v_{n-2} & w_{n-2} & -v_{n-2} & -w_{n-2} & 0 & \cdots & 0
\end{pmatrix} : 
\begin{array}{l} a,x,y \in \RR,\\ 
v,w \in \RR^{n-2} \end{array}
\right\}
\]
of $\ca \oplus \cn$. Recall that $\ca^1 \oplus \cg_{\alpha_1} = \RR H_{\alpha_1} \oplus \cg_{\alpha_1}$ generates the boundary component $B_1 \cong \CC H^1(-4)$. The orbit containing $o$ of the $1$-dimensional Lie group generated by $\cg_{\alpha_1}$ is a horocycle in the boundary component $B_1$. Since the tangent vectors of $B_1$ are $\cA$-isotropic, the horocycle is $\cA$-isotropic. The canonical extension of this cohomogeneity one action on $B_1$ is the cohomogeneity one action on ${Q^n}^*$ by the subgroup $H_1$ of $AN$ with Lie algebra $\ch_1$. Let $\hat{M}^{2n-1}_2 = H_1 \cdot o \cong H_1$ be the orbit of the $H_1$-action on $(AN,\langle \cdot , \cdot \rangle)$ containing $o$ and $M^{2n-1}_2 = H_1 \cdot o$ be the orbit of the $H_1$-action on $({Q^n}^*,g)$ containing $o$. 

The normal space $\nu_o\hat{M}^{2n-1}_2$ of $\hat{M}^{2n-1}_2$ at $o$ is
\[
\nu_o\hat{M}^{2n-1}_2 = \ca^1 = \RR H_{\alpha_1}.
\]
Since $\langle H_{\alpha_1} , H_{\alpha_1} \rangle = 4$, the vector
$\hat\zeta = \frac{1}{2}H_{\alpha_1} \in \ca$ 
is a unit normal vector of $\hat{M}^{2n-1}_2$ at $o$. Let $\hat{A}$ be the shape operator of $\hat{M}^{2n-1}_2$ in $(AN,\langle \cdot , \cdot \rangle)$ with respect to $\hat{\zeta}$. As in previous sections, we can show that the shape operater $\hat{A}$ of $\hat{M}^{2n-1}_2$ is given by
\[
\hat{A} \hat{X} = [\zeta,\hat{X}]_{\ch_1}
\]
for all $\hat{X} \in \ch_1$,
where \[
\zeta = \frac{1}{2}(\hat\zeta - \theta(\hat\zeta)) = \hat\zeta = \frac{1}{2}H_{\alpha_1}
\] 
and $[\, \cdot\ ]_{\ch_1}$ is the orthogonal projection onto $\ch_1$.

For
\[
\hat{X} = \begin{pmatrix}
0 & x+y & a & x-y & v_1 & \cdots & v_{n-2} \\
-x-y & 0 & x+y & a & w_1 & \cdots & w_{n-2} \\
a & x+y & 0 & x-y & v_1 & \cdots & v_{n-2} \\
x-y & a & -x+y & 0 & w_1 & \cdots & w_{n-2} \\
v_1 & w_1 & -v_1 & -w_1 & 0 & \cdots & 0 \\
\vdots & \vdots & \vdots & \vdots & \vdots & \ddots & \vdots\\
v_{n-2} & w_{n-2} & -v_{n-2} & -w_{n-2} & 0 & \cdots & 0
\end{pmatrix} \in \ch_1
\]
we then compute
\[
[\zeta,\hat{X}]  =  \begin{pmatrix}
0 & 2x & 0 & 2x & v_1 & \cdots &  v_{n-2} \\
-2x & 0 & 2x & 0 & -w_1 & \cdots &  -w_{n-2} \\
0 & 2x & 0 & 2x & v_1 & \cdots &  v_{n-2}\\
2x & 0 & -2x & 0 & -w_1 & \cdots &  -w_{n-2} \\
v_1 & -w_1 & -v_1 & w_1 & 0 & \cdots & 0\\
\vdots & \vdots & \vdots & \vdots & \vdots & \ddots & \vdots  \\
v_{n-2} & -w_{n-2} & -v_{n-2} & w_{n-2} & 0 & \cdots & 0  
\end{pmatrix}.
\]
Since the last matrix is in $\ch_1$, the orthogonal projection of $[\zeta,\hat{X}]$ onto $\ch_1$ is $[\zeta,\hat{X}]$.
We conclude that the shape operator $\hat{A}$ of $\hat{M}^{2n-1}_2$ in $(AN,\langle \cdot , \cdot \rangle)$ is given by
\[
\hat{A} \hat{X} = 
\begin{pmatrix}
0 & 2x & 0 & 2x & v_1 & \cdots &  v_{n-2} \\
-2x & 0 & 2x & 0 & -w_1 & \cdots &  -w_{n-2} \\
0 & 2x & 0 & 2x & v_1 & \cdots &  v_{n-2}\\
2x & 0 & -2x & 0 & -w_1 & \cdots &  -w_{n-2} \\
v_1 & -w_1 & -v_1 & w_1 & 0 & \cdots & 0\\
\vdots & \vdots & \vdots & \vdots & \vdots & \ddots & \vdots  \\
v_{n-2} & -w_{n-2} & -v_{n-2} & w_{n-2} & 0 & \cdots & 0
\end{pmatrix}
\]
with
\[
\hat{X} = \begin{pmatrix}
0 & x+y & a & x-y & v_1 & \cdots & v_{n-2} \\
-x-y & 0 & x+y & a & w_1 & \cdots & w_{n-2} \\
a & x+y & 0 & x-y & v_1 & \cdots & v_{n-2} \\
x-y & a & -x+y & 0 & w_1 & \cdots & w_{n-2} \\
v_1 & w_1 & -v_1 & -w_1 & 0 & \cdots & 0 \\
\vdots & \vdots & \vdots & \vdots & \vdots & \ddots & \vdots\\
v_{n-2} & w_{n-2} & -v_{n-2} & -w_{n-2} & 0 & \cdots & 0
\end{pmatrix} \in \ch_1.
\]

From this we deduce that the principal curvatures of $\hat{M}^{2n-1}_2$ are $2$, $0$, $1$, $-1$ with corresponding multiplicities $1$, $2$, $n-2$, $n-2$, respectively. The corresponding principal curvature spaces are
\[
\hat{T}_2   = \cg_{\alpha_1}\ ,\ 
\hat{T}_0  = \ca_1 \oplus \cg_{\alpha_1\oplus 2\alpha_2} \ ,\ 
\hat{T}_1  = \cg_{\alpha_1+\alpha_2}\ ,\ 
\hat{T}_{-1}  = \cg_{\alpha_2}.
\]

All of the above calculations are with respect to the metric $\langle \cdot , \cdot \rangle$ on $AN$. We now switch to the Riemannian metric $g$ on ${Q^n}^*$ and the Cartan decomposition $\cg = \ck \oplus \cp$. Recall that, by construction, $(AN,\langle \cdot , \cdot \rangle)$ and $({Q^n}^*,g)$ are isometric and the metrics are related by
\[
\langle H_1 + \hat{X}_1 , H_2 + \hat{X}_2 \rangle  
= g(H_1,H_2) + g(X_1,X_2)
\]
with $H_1,H_2 \in \ca$ and $\hat{X}_1,\hat{X}_2 \in \cn$.

Let $A$ be the shape operator of $M^{2n-1}_2$ in $({Q^n}^*,g)$ with respect to $\zeta$. The above calculations then imply
\[
A X = 
 \begin{pmatrix}
0 & 0 & 0 & 2x & v_1 & \cdots &  v_{n-2} \\
0 & 0 & 2x & 0 & -w_1 & \cdots &  -w_{n-2} \\
0 & 2x & 0 & 0 &  0 & \cdots & 0\\
2x & 0 & 0 & 0 &  0 & \cdots & 0 \\
v_1 & -w_1 & 0 & 0 & 0 & \cdots & 0\\
\vdots & \vdots & \vdots & \vdots & \vdots & \ddots & \vdots  \\
v_{n-2} & -w_{n-2} & 0 & 0 & 0 & \cdots & 0
\end{pmatrix}
\]
with
\[
X = \begin{pmatrix}
0 & 0 & a & x-y & v_1 & \cdots & v_{n-2} \\
0 & 0 & x+y & a & w_1 & \cdots & w_{n-2} \\
a & x+y & 0 & 0 & 0 & \cdots & 0 \\
x-y & a & 0 & 0 & 0 & \cdots & 0 \\
v_1 & w_1 & 0 & 0 & 0 & \cdots & 0 \\
\vdots & \vdots & \vdots & \vdots & \vdots & \ddots & \vdots\\
v_{n-2} & w_{n-2} & 0 & 0 & 0 & \cdots & 0
\end{pmatrix} \in T_oM^{2n-1}_2 \subset \cp.
\]
From this we  deduce the following result.

\begin{thm} 
The homogeneous real hypersurface $M^{2n-1}_2$ in $({Q^n}^*,g)$ has four distinct constant principal curvatures $2$, $0$, $1$, $-1$ with multiplicities $1$, $2$, $n-2$, $n-2$, respectively. The principal curvature spaces $T_2$, $T_0$, $T_1$ and $T_{-1}$ are
\[
T_2   = \cp_{\alpha_1} = \RR J\zeta,\ 
T_0  = \RR H_{\alpha+2\alpha_2} \oplus \cp_{\alpha_1\oplus 2\alpha_2} = \calC \ominus \calQ  ,\ 
T_1  = \cp_{\alpha_1+\alpha_2},\ 
T_{-1}  = \cp_{\alpha_2}.
\]
In particular, $T_1$ and $T_{-1}$ are mapped into each other by the structure tensor field $\phi$.
Moreover, the shape operator $A$ and the structure tensor field $\phi$ of $M^{2n-1}_2$ satisfy
\[
A \phi + \phi A = 0.
\]
\end{thm}

Note that $T_1 \subset V(C_0)$ and $T_{-1} \subset JV(C_0)$.

The eigenspaces $T_2$, $T_0$, $T_1$ and $T_{-1}$ of the shape operator $A$ and the eigenspaces $E_0$, $E_{-1}$ and $E_{-4}$ of the normal Jacobi operator $\calK = \calK_\zeta$ satisfy
\[
T_2 = E_{-4} ,\ T_0 = E_0  ,\ T_{-1} \oplus T_1 = E_{-1}.
\]
It follows that $A$ and $\calK$ are simultaneously diagonalizable and hence $A\calK = \calK A$. Thus we have proved the following.

\begin{cor} 
The homogeneous Hopf hypersurface $M^{2n-1}_2$ in $({Q^n}^*,g)$ is curvature-adapted.
\end{cor}

By construction, the integral curves of the Reeb vector field $\xi$ are congruent to a horocycle in a complex hyperbolic line $\CC H^1(-4)$. Such a horocycle has constant geodesic curvature $2$. As in previous cases, this leads to the geometric interpretation of $M^{2n-1}_2$ being obtained by attaching isometric copies of the homogeneous complex hypersurface $P^{n-1} \cong \CC H^{n-1}(-4)$ to the horocycle in a suitable way. Equivalently, $M^{2n-1}_2$ is the canonical extension of a horocycle in the boundary component $B_1 \cong \CC H^1(-4)$.

\section{Curvature} \label{curvature}

In this section we compute the Ricci tensor $\Ric_\alpha$ and the scalar curvature $s_\alpha$ of the homogeneous Hopf hypersurface $M^{2n-1}_\alpha$ in $({Q^n}^*,g)$. Let $R_\alpha$, $\Ric_\alpha$, $s_\alpha$ be the Riemannian curvature tensor, Ricci tensor, scalar curvature of $M^{2n-1}_\alpha$, respectively. Let $A_\alpha$ and $\calK_\alpha$ be the shape operator and normal Jacobi operator of $M^{2n-1}_\alpha$ with respect to the unit normal vector $\zeta_\alpha$, respectively. The Gauss equation tells us that
\[
g(\bar{R}(X,Y)Z, W) = g(R_\alpha(X, Y)Z,W) - g(A_\alpha Y,Z)g(A_\alpha X,W) + g(A_\alpha X,Z)g(A_\alpha Y,W)
\]
for all $X,Y,Z,W \in \cX(M^{2n-1}_\alpha)$. 
Contracting the Gauss equation gives, after some straightforward computations, the expression
\[
\Ric_\alpha X = -2nX - \calK_\alpha X + \alpha A_\alpha X - A_\alpha^2X,
\]
where we used the fact that the Ricci tensor of $({Q^n}^*,g)$ is equal to $-2ng$ and $\tr(A_\alpha) = \alpha$ by Theorem \ref{mainthm}. Since the unit normal vector $\zeta_\alpha$ of $M^{2n-1}_\alpha$ is $\cA$-isotropic, the normal Jacobi operator $\calK_\alpha$ of $M^{2n-1}_\alpha$ satisfies
\[
\calK_\alpha X = 
\begin{cases}
0 & ,\mbox{ if } X \in \calC \ominus \calQ = T_0, \\
-X & ,\mbox{ if } X \in \calQ = T_{-1} \oplus T_1 ,\\
-4X & ,\mbox{ if } X \in \calC^\perp = \RR\xi = T_\alpha
\end{cases}
\]
by Theorem \ref{mainthm} and the description of the Jacobi operator  in Section \ref{tchq}. It follows that 
\[
\Ric_\alpha X = 
\begin{cases}
-2nX & ,\mbox{ if } X \in \calC \ominus \calQ = T_0, \\
(-2n-\alpha)X & ,\mbox{ if } X \in T_{-1} , \\
(-2n+\alpha)X & ,\mbox{ if } X \in T_1 , \\
(-2n+4)X & ,\mbox{ if } X \in \calC^\perp = \RR\xi = T_\alpha.
\end{cases}
\]
It follows that the Ricci tensor of $M^{2n-1}_\alpha$ has two (if $\alpha = 0$), three (if $\alpha = 4$) or four (if $\alpha \notin \{0,4\}$) constant eigenvalues. More specifically, for $\alpha = 0$ we obtain
\[
\Ric_0 X  = -2n X + 4\eta(X)\xi,
\]
which means that $M^{2n-1}_0$ is pseudo-Einstein (see \cite{Ko79}). 

\begin{prop}
The minimal homogeneous real hypersurface $M^{2n-1}_0$ is a pseudo-\break Einstein Hopf hypersurface in $({Q^n}^*,g)$. In particular, the Ricci tensor $\Ric_0$ of $M^{2n-1}_0$ is $\phi$-invariant, that is, $\Ric_0 \circ \phi = \phi \circ \Ric_0$.
\end{prop}

We also see that
\[
\Ric_\alpha \circ \phi + \phi \circ \Ric_\alpha = -4n \phi .
\]
This equation is motivated by Ricci solitons (see \cite{BS22}, Lemma 3.3.11). However, none of the homogeneous Hopf hypersurfaces $M^{2n-1}_\alpha$ is a Ricci soliton.

By contracting the Ricci tensor we see that the scalar curvature of $M^{2n-1}_\alpha$ is independent of $\alpha$. 

\begin{prop}
The scalar curvature $s_\alpha$ of the homogeneous Hopf hypersurface $M^{2n-1}_\alpha$ in $({Q^n}^*,g)$ does not depend on $\alpha$ and satisfies 
\[
s_\alpha = 4-2n(2n-1).
\]
\end{prop}

\end{document}